\documentclass[leqno,11pt]{amsart}

\usepackage{amsmath,amstext,amssymb,amsopn,amsthm,mathrsfs}
\usepackage{verbatim}
\title[Operators related to symmetrized Jacobi expansions]{Harmonic analysis operators related to symmetrized Jacobi expansions}

\author[B{.} Langowski]{Bartosz Langowski}

\address{Bartosz Langowski, \newline
			Institute of Mathematics and Computer Science, \newline
      Wroc\l{}aw University of Technology,       \newline
      Wyb{.} Wyspia\'nskiego 27,
      50--370 Wroc\l{}aw, Poland      
      }
\email{bartosz.langowski@pwr.wroc.pl}

\allowdisplaybreaks
\newcommand{\el}{l}
\theoremstyle{plain}
\newtheorem{thm}{Theorem}[section]

\newtheorem{lem}[thm]{Lemma}
\newtheorem{prop}[thm]{Proposition}
\newtheorem{cor}[thm]{Corollary}
\newtheorem{defin}[thm]{Definition}

\theoremstyle{definition}

\theoremstyle{remark}
\newtheorem*{rem*}{Remark}

\theoremstyle{plain}

\DeclareMathOperator{\domain}{Dom}
\DeclareMathOperator{\spectrum}{Spec}
\DeclareMathOperator{\support}{supp}

\DeclareMathOperator*{\essup}{ess\,sup}

\def\ko{\delta_N^{\textrm{odd}}}

\def\d{d}
\def\ep{\epsilon}
\def\P{\mathcal P}

\def\m{\mu} 						
\def\ab{\alpha,\beta}
\def\J{\mathcal J} 			
\def\q{\mathfrak q}
\def\de{\delta_N^{\textrm{even}}}

\begin{document}

\begin{abstract}
Following a symmetrization procedure proposed recently by Nowak and Stempak, we consider the setting of symmetrized Jacobi expansions. In this framework we investigate mapping properties of several fundamental harmonic analysis operators, including Riesz transforms, Poisson semigroup maximal operator, Littlewood-Paley-Stein square functions and multipliers of Laplace and Laplace-Stieltjes transform type. Our paper delivers also some new results in the original setting of classical Jacobi expansions.
\end{abstract}

\renewcommand{\today}

\maketitle

\section{Introduction} \label{sec:intro}

This research is motivated by the recent results of Nowak and Sj\"ogren \cite{NoSjogren} and Nowak and Stempak \cite{Symmetrized}. In  \cite{NoSjogren} the authors investigated several fundamental operators in harmonic analysis related to classical Jacobi expansions, including Riesz transforms, imaginary powers of the Jacobi operator, the Jacobi-Poisson semigroup maximal operator and Littlewood-Paley-Stein type square functions. These operators were shown to be (vector-valued) Calder\'on-Zygmund operators in the sense of the associated space of homogeneous type and hence their basic mapping properties could be concluded from the general theory. The results of  \cite{NoSjogren} embrace a continuation and extension of the investigations undertaken in the seminal work of Muckenhoupt and Stein \cite{MuS} and succeeding papers by other authors, see \cite[Section 1]{NoSjogren} and also the comments closing \cite[Section 2]{NoSjogren}. On the other hand, in \cite{Symmetrized} the authors proposed a symmetrization procedure in a context of general orthogonal expansions associated with a second order differential operator $L$, a `Laplacian'. This procedure, combined with a unified conjugacy scheme established in another article by Nowak and Stempak \cite{L2}, allows one to associate, via a suitable embedding, a differential-difference `Laplacian' $\mathbb{L}$ with the initially given orthogonal system of eigenfunctions of $L$ so that the resulting extended conjugacy scheme has the natural classical shape. In particular, the related `partial derivatives' decomposing   $\mathbb{L}$ are skew-symmetric in appropriate $L^2$ space and they commute with Riesz transforms and conjugate Poisson integrals. Thus the symmetrization procedure overcomes the main inconvenience of the theory postulated in \cite{L2}, that is the lack of symmetry in the principal objects and relations resulting in essential deviations of the theory from the classical shape. The price is, however, that the `Laplacian' $\mathbb{L}$ and the associated `derivatives' are not differential, but difference-differential operators. The asymmetry just mentioned concerns, in particular, the point of view taken in \cite{NoSjogren}.

The main objective of the present work is to apply the symmetrization procedure to the setting of classical Jacobi expansions considered in \cite{NoSjogren} and study $L^p$ theory of fundamental harmonic analysis operators in the symmetrized framework. One of natural questions motivating this research is whether the symmetrization procedure, proved in \cite{Symmetrized} to be reasonable from the $L^2$ theory perspective, is also supported by the associated $L^p$ theory. The answer we get is affirmative. Another important aspect of our results is related to the fact that most definitions and relations in the initial non-symmetrized setting may be recovered as suitable `projections' of the corresponding objects in the extended symmetrized situation. Therefore the multiplier results obtained in this paper imply readily new multiplier theorems for Jacobi expansions (recall that multiplier operators, as such, were not considered in \cite{NoSjogren}). Noteworthy, the multipliers in question cover, as special cases, imaginary and fractional powers of the Jacobi `Laplacian'. However, in some cases the `projections' are different from the corresponding non-symmetrized objects existing in the literature. Nevertheless, they appear to be even more appropriate and natural, see comments in \cite[Section 5]{Symmetrized}. This remark concerns especially higher order Riesz transforms and mixed $g$-functions of higher orders. Consequently, our present results lead directly to analogous new results for alternatively defined higher order Riesz transforms and mixed square functions in the Jacobi setting, which differ from those investigated in \cite{NoSjogren}. A remarkable potential application of the latter results is connected with an open problem of developing the theory of Sobolev spaces related to classical Jacobi expansions.

Let $\alpha,\beta> -1$. Recall that the normalized Jacobi trigonometric polynomials $\P_n^{\ab}$ (see Section \ref{sec:prel} for the definition) form an orthonormal basis in $L^2$ of the interval $(0,\pi)$ with respect to the measure 
$$d\m_{\ab}^{+}(\theta) = \Big(\sin\frac{\theta}2\Big)^{2\alpha+1} 
	\Big( \cos\frac{\theta}2\Big)^{2\beta+1} d\theta,\qquad \theta\in (0,\pi),$$
and they are eigenfunctions of the Jacobi operator 
\begin{equation}\label{jacobi}
\J^{\ab} = - \frac{d^2}{d\theta^2} - \frac{\alpha-\beta+(\alpha+\beta+1)\cos\theta}{\sin \theta}
	\frac{d}{d\theta} + \Big( \frac{\alpha+\beta+1}{2}\Big)^2.
\end{equation}
The factorization 
$$
\J^{\ab} = \delta^*\delta + \Big( \frac{\alpha+\beta+1}{2}\Big)^2,
$$
where $\delta = {d}\slash{d\theta}$ and $\delta^* = -\frac{d}{d\theta} - (\alpha+\frac{1}2)\cot\frac{\theta}2+(\beta+\frac{1}{2})\tan\frac{\theta}2$ is the formal adjoint of $\delta$ in $L^2(d\m_{\ab}^{+}),$ determines a natural derivative $\delta$ associated with $\J^{\ab}.$ Notice that $\delta\neq-\delta^*$ in general. In this paper we shall work principally on the larger interval $(-\pi,\pi)$ equipped with the measure 
$$
d\m_{\ab}(\theta) = \Big|\sin\frac{\theta}2\Big|^{2\alpha+1} 
	\Big( \cos\frac{\theta}2\Big)^{2\beta+1} d\theta,\qquad \theta\in (-\pi,\pi).
$$
The system of symmetrized Jacobi polynomials $\{\Phi_n^{\ab}: n\ge 0\}$ (see Section \ref{sec:prel}) is an orthonormal basis in $L^2(d\m_{\ab})$ consisting of eigenfunctions of the symmetrized Jacobi `Laplacian'
\begin{align}\label{symml}
\mathbb{J}^{\ab}f & = -\frac{d^2f}{d\theta^2} - \frac{\alpha-\beta+(\alpha+\beta+1)\cos\theta}{\sin\theta}
	\frac{df}{d\theta} + \Big(\frac{\alpha+\beta+1}2\Big)^2 f\\ \nonumber
	& \quad + 
	\frac{(\alpha+\beta+1)+ (\alpha-\beta)\cos\theta}{\sin^2\theta} f_{\textrm{odd}},
\end{align}
where $f_{\textrm{odd}}$ is the odd part of $f$. This operator can be decomposed as
$$\mathbb{J}^{\ab}=-D^2+\Big(\frac{\alpha+\beta+1}{2}\Big)^2,$$
where the `derivative' $D$ given by
$$Df=\frac{df}{d\theta}+\frac{\alpha-\beta+(\alpha+\beta+1)\cos\theta}{\sin\theta}f_{\textrm{odd}}$$
is formally skew-adjoint in $L^2(d\m_{\ab})$. It is worth to note that this $D$ coincides with the Jacobi-Dunkl operator on the interval $(-\pi,\pi)$, and the system $\{\Phi_n^{\ab}: n\ge 0\}$ consists essentially of so-called Jacobi-Dunkl polynomials (see the comment in \cite[Example 4]{Symmetrized}). However, the situation considered here is not isomorphic to the Jacobi-Dunkl setting since the `Laplacian' emerging from the symmetrization procedure differs somewhat from the Jacobi-Dunkl `Laplacian' and consequently harmonic analysis operators studied in this paper are not the same as their Jacobi-Dunkl counterparts.

The central objects of our study are the following linear or sublinear operators associated with $\mathbb{J}^{\ab}$ (strict definitions will be given in Section \ref{sec:prel}).
\begin{itemize}

\item[(i)] Symmetrized Riesz-Jacobi transforms of arbitrary order $N$
$$
\mathbb{R}_N^{\ab}\colon f \mapsto D^{N} \big(\mathbb{J}^{\ab}\big)^{-N\slash 2}f, \qquad N=1,2,\ldots.
$$
\item[(ii)] The symmetrized Jacobi-Poisson semigroup maximal operator
$$
\mathbb{H}_{*}^{\ab}\colon f\mapsto \big\| \exp\big(-t\sqrt{\mathbb{J}^{\ab}}\big)f\big\|_{L^{\infty}(dt)}.
$$

\item[(iii)] Mixed square functions of arbitrary orders $M,N$
$$
\mathbb{G}_{M,N}^{\ab}\colon f \mapsto\big\| \partial_t^{M} D^{N} 
\exp\big(-t\sqrt{\mathbb{J}^{\ab}}\big)f \big\|_{L^2(t^{2M+2N-1}dt)},
$$
where $M,N=0,1,2,\ldots$ and $M+N>0$. 
\item[(iv)] Laplace and Laplace-Stieltjes transform type multipliers
$$
\mathbb{M}^{\ab}\colon f \mapsto m\Big(\sqrt{\mathbb{J}^{\ab}}\Big)f,
$$
where $m$ is as in Definition \ref{def} below.
\end{itemize}

\begin{defin}\label{def}
Following E{.} M{.} Stein \cite[p.\,58,\,p.\,121]{Ste} we say that $m$ is a multiplier of Laplace (transform) type associated with $\sqrt{\mathbb{J}_{\ab}}$ if $m$ has the form
\begin{align*}
m(z)=m_{\phi}(z)=\int_0^{\infty} z e^{-tz}\phi(t)\,dt, \qquad z \ge \Big|\frac{\alpha+\beta+1}{2}\Big|,
\end{align*} 
where $\phi$ is a bounded measurable function on $(0,\infty)$. Inspired by \cite{drel,blazej,szarek2}, we also consider multipliers of Laplace-Stieltjes transform type associated with $\sqrt{\mathbb{J}_{\ab}}$ and having the form
\begin{align*}
m(z)=m_{\nu}(z)=\int_{(0,\infty)} e^{-tz}\,d\nu(t), \qquad z \ge \Big|\frac{\alpha+\beta+1}{2}\Big|,
\end{align*} 
where $\nu$ is a signed or complex Borel measure on $(0,\infty)$ with total variation $|\nu|$ satisfying
\begin{align*}
\int_{(0,\infty)} e^{-t|\alpha+\beta+1|/2}\,d|\nu|(t)<\infty.
\end{align*} 
\end{defin}
 
 Our main result is contained in Theorem \ref{thm:main}. It says that under the slight restriction $\alpha,\beta\ge -1\slash 2$ the operators (i)-(iv) are bounded
on $L^p(wd\m_{\ab})$, $1<p<\infty$, and from $L^1(wd\m_{\ab})$ to weak $L^1(wd\m_{\ab})$ for a reasonably large class of even weights $w$ on $(-\pi,\pi)$. When restricted to even functions, this implies analogous results in the (non-symmetrized) Jacobi setting that are new in cases of the Laplace and Laplace-Stieltjes type multipliers and alternatively defined higher order Riesz-Jacobi transforms and mixed $g$-functions. Thus Theorem \ref{thm:main} may be regarded as an extension and completion of the results obtained in \cite{NoSjogren}.

The proof of Theorem \ref{thm:main} proceeds as follows. Exploiting the arguments from \cite{imag}, in the first step we reduce the task to showing similar results, stated in Theorem \ref{thm:main'}, for suitably defined operators related to the smaller measure space $((0,\pi),d\mu_{\ab}^{+})$. Then Theorem \ref{thm:main'} is proved by means of the general (vector-valued) Calder\'on-Zygmund theory with the underlying space of homogeneous type $((0,\pi),d\mu_{\ab}^{+},|\cdot|),$ where $|\cdot|$ stands for the ordinary Euclidean distance. The same approach was used in \cite{NoSjogren} and hence we can take a direct advantage of some results obtained there. The most troublesome part of our line of reasoning is to show that certain integral kernels satisfy so-called standard estimates. An essential tool to perform this task is a suitable convenient representation of the Jacobi-Poisson kernel obtained in \cite[Proposition 4.1]{NoSjogren} as an indirect consequence of the product formula for Jacobi polynomials due to Dijksma and Koornwinder \cite{DK}. Note that the latter is valid only for $\alpha,\beta\ge-1/2$, and this restriction is inherited by the present results. Other important tools are delivered by the elegant technique of estimating kernels developed in \cite[Section 4]{NoSjogren}.

The paper is organized as follows. Section \ref{sec:prel} contains the setup, strict definitions of the operators (i)-(iv) and statement of the main result (Theorem \ref{thm:main}). Also, suitable Jacobi-type operators, related to the restricted space $((0,\pi),d\mu_{\ab}^{+})$, are defined and the proof of Theorem \ref{thm:main} is reduced to showing that these auxiliary operators can be interpreted as Calder\'on-Zygmund operators. In Section \ref{sec:L^2} the Jacobi-type operators are proved to be $L^2$-bounded and associated, in the Calder\'on-Zygmund theory sense, with the relevant kernels. This part is quite standard and some of the proofs are similar to those already examined in different settings, but we present rather complete reasoning. Finally, Section \ref{sec:ker} is the core of this work and is devoted to the proofs of all necessary kernel estimates.

Throughout the paper we use fairly standard notation, with all symbols referring either to the measure space $((-\pi,\pi),d\mu_{\ab})$ or to the space of homogeneous type $((0,\pi),d\mu_{\ab}^{+},|\cdot|)$, depending on the context. In the latter case the ball denoted $B(\theta,r)$ is simply the interval $(\theta-r,\theta+r)\cap (0,\pi)$.
By $\langle f,g\rangle_{d\m_{\ab}}$ we mean $\int_{-\pi}^{\pi} f(\theta)\overline{g(\theta)}\, d\m_{\ab}(\theta)$ whenever the integral makes sense, and similarly for $\langle f,g\rangle_{d\m_{\ab}^{+}}$. Further, $L^p(wd\m_{\ab})$ or $L^p(wd\m_{\ab}^{+})$ stand for the weighted $L^p$ spaces, $w$ being a nonnegative
weight on $(-\pi,\pi)$ or $(0,\pi)$, respectively. Given $1 \le p < \infty$, $p'$ is its adjoint exponent, $1\slash p + 1\slash p' =1$.
The Muckenhoupt class of $A_p$ weights related to the measure $d\m_{\ab}$ will be denoted by $A^{\ab}_p$.  
More precisely, $A^{\ab}_p$ is the class of all nonnegative functions $w$ such that
\begin{equation*}
\sup_{I \in \mathcal{I}} \bigg[ \frac{1}{\m_{\ab}(I)} \int_I w(\theta) \,
     d\m_{\ab}(\theta) \bigg]  \bigg[ \frac{1}{\m_{\ab}(I)} \int_I w(\theta)^{-p'\slash p}
     \, d\m_{\ab}(\theta) \bigg]^{p\slash p'} < \infty
\end{equation*}
when $1<p<\infty$, or
\begin{equation*}
\sup_{I \in \mathcal{I}} \frac{1}{\m_{\ab}(I)} \int_I w(\theta) \,
     d\m_{\ab}(\theta) \; \essup_{\theta \in I} \frac{1}{w(\theta)} < \infty
\end{equation*}
if $p=1$; here $\mathcal{I}$ is the family of all subintervals of $(-\pi,\pi).$
The definition of the Muckenhoupt class $(A^{\ab}_p)^{+}$ related to $d\mu_{\ab}^{+}$ is analogous. Note that $A^{\ab}_p\subset L^{1}(d\mu_{\ab})$ and $(A^{\ab}_p)^{+}\subset L^{1}(d\mu_{\ab}^{+}).$ It is easy to check that an even weight $w\in A^{\ab}_p$ if and only if $w^{+}\in (A^{\ab}_p)^{+}.$ Moreover, a double-power weight $w(\theta)=\big|\sin\frac{\theta}{2}\big|^r\big(\cos\frac{\theta}{2}\big)^s$ belongs to $ A^{\ab}_p$, $1<p<\infty$, if and only if $-(2\alpha+2)<r<(2\alpha+2)(p-1)$ and $-(2\beta+2)<s<(2\beta+2)(p-1)$, and $w\in A^{\ab}_1$ if and only if $-(2\alpha+2)<r\le 0$ and $-(2\beta+2)<r\le 0.$ Finally, given a function $f$ on $(-\pi,\pi)$, we will denote by $f_{\textrm{even}}$ and $f_{\textrm{odd}}$ its even and odd part, respectively,
 $$f_{\textrm{even}}(\theta)=\frac{f(\theta)+f(-\theta)}{2},\qquad f_{\textrm{odd}}(\theta)=\frac{f(\theta)-f(-\theta)}{2},$$
and by $f^{+}$ its restriction to the interval $(0,\pi)$.

While writing estimates, we will frequently use the notation $X \lesssim Y$
to indicate that $X \le CY$ with a positive constant $C$ independent of significant quantities.
We shall write $X \simeq Y$ when simultaneously $X \lesssim Y$ and $Y \lesssim X$.

\textbf{Acknowledgments.}
The author would like to express his gratitude to Professor Adam Nowak for indicating the topic and constant support during the preparation of this paper and to Tomasz Z. Szarek for some valuable remarks.

\section{Preliminaries and statement of main results} \label{sec:prel}

Given $\alpha, \beta > -1$, the standard Jacobi polynomials
of type $\ab$ are defined on the interval $(-1,1)$ by the Rodrigues formula (cf. \cite[(4.3.1)]{Sz})
$$
P_n^{\ab} (x) = \frac{(-1)^n}{2^n n!} (1-x)^{-\alpha}(1+x)^{-\beta}
    \frac{{\d}^n}{{\d}x^n} \big( (1-x)^{\alpha+n}(1+x)^{\beta+n} \big),
    \qquad n=0,1,2,\ldots.
$$
Note that each $P_n^{\ab}$ is a polynomial of degree $n$.
It is natural and convenient to apply the trigonometric parametrization 
$x=\cos\theta$, $\theta\in (0,\pi)$, and consider the normalized trigonometric polynomials
$$
\P_n^{\ab}(\theta) = c_n^{\ab} P_n^{\ab}(\cos\theta),
$$
with the normalizing factor
$$
c_n^{\ab}=\|P_n^{\ab}\circ\cos\|^{-1}_{L^2(d\m_{\ab}^{+})} =
	\bigg( \frac{(2n+\alpha+\beta+1)\Gamma(n+\alpha+\beta+1)\Gamma(n+1)}
		{\Gamma(n+\alpha+1)\Gamma(n+\beta+1)} \bigg)^{1\slash 2},
$$
where for $n=0$ and $\alpha+\beta=-1$
the product $(2n+\alpha+\beta+1)\Gamma(n+\alpha+\beta+1)$ must be replaced by
$\Gamma(\alpha+\beta+2)$. It is well known that the system $\{\P_n^{\ab}:n\ge 0\}$
is orthonormal and complete in $L^2(d\m_{\ab}^{+})$. 
Moreover, each $\P_n^{\ab}$ is an eigenfunction of the Jacobi operator \eqref{jacobi}, 
$$
\J^{\ab} \P_n^{\ab} = \lambda_n^{\ab} \P_n^{\ab}, 
	\qquad \lambda_n^{\ab}=\Big( n+ \frac{\alpha+\beta+1}{2}\Big)^2.
$$
Thus $\J^{\ab}$, considered initially on $C_c^{2}(0,\pi)$, has a natural self-adjoint
extension in $L^2(d\m_{\ab}^{+})$, still denoted by the same symbol $\J^{\ab}$, and given by
\begin{equation} \label{sres'}
\J^{\ab} f = \sum_{n=0}^{\infty} \lambda_n^{\ab} \langle f, \P_n^{\ab} \rangle_{d\m_{\ab}^{+}} \P_n^{\ab}
\end{equation}
on the domain $\domain\J^{\ab}$ consisting of all functions $f\in L^2(d\m_{\ab}^{+})$ 
for which the defining series converges in $L^2(d\m_{\ab}^{+})$. 
Then the spectral decomposition of $\J^{\ab}$ is given by \eqref{sres'}.

The symmetrization procedure was described in \cite[Section 3]{Symmetrized}. Applied to the setting of Jacobi trigonometric expansions on the interval $(0,\pi)$, it leads to the following symmetrized framework on the interval $(-\pi,\pi)$. The symmetrized Jacobi operator $\mathbb{J}^{\ab}$ is given by \eqref{symml}. The system $\{\Phi_n^{\ab}: n\ge 0\}$ on $(-\pi,\pi)$ associated with $\mathbb{J}^{\ab}$ and related to the original
system $\{\P_n^{\ab}:n\ge 0\}$ is defined as
\begin{align}\label{Phi}
\Phi_{n}^{\ab} =
\begin{cases}
\frac{1}{\sqrt{2}} \P_{n/2}^{\ab}, & n \;\; \textrm{even} \\
-\frac{1}{\sqrt{2}} \big(\lambda_{(n+1)\slash 2}^{\alpha,\beta}-\lambda_{0}^{\alpha,\beta}\big)^{-1\slash 2} 
	\, \delta\P_{(n+1)/2}^{\ab}, & n \;\; \textrm{odd}
\end{cases},
\end{align}
where $\P_k^{\ab}$ are considered as functions on $(-\pi,\pi).$ According to \cite[Lemma 3.4]{Symmetrized}, each $\Phi_{n}^{\ab}$ is an eigenfunction of the symmetrized Jacobi operator. More precisely,
$$
\mathbb{J}^{\alpha,\beta}\Phi_n^{\ab} = \lambda_{\langle n \rangle} \Phi_n^{\ab}, \qquad n \in \mathbb{N},
$$
where we use the notation $\langle n \rangle=\Big\lfloor\frac{n+1}2 \Big\rfloor$ introduced in \cite{Symmetrized} (here $\lfloor\cdot\rfloor$ denotes the floor function). Thus $\mathbb{J}^{\ab}$, considered initially on $C_c^{2}((-\pi,0)\cup(0,\pi))$, has a natural self-adjoint extension in $L^2(d\m_{\ab})$, still denoted by $\mathbb{J}^{\ab}$ and given by
\begin{equation} \label{sresbb}
\mathbb{J}^{\ab} f = \sum_{n=0}^{\infty} \lambda_{\langle n \rangle}^{\ab} \langle f, \Phi_n^{\ab} \rangle_{d\m_{\ab}} \Phi_n^{\ab}
\end{equation}
on the domain $\domain\mathbb{J}^{\ab}$ consisting of all functions $f\in L^2(d\m_{\ab})$ 
for which the defining series converges in $L^2(d\m_{\ab})$; see \cite[Section 4]{Symmetrized}. Clearly, the spectral decomposition of $\mathbb{J}^{\ab}$ is given by \eqref{sresbb}.

The semigroup of operators generated by the square root of $\mathbb{J}^{\ab}$ will be called the symmetrized Jacobi-Poisson semigroup
and will be denoted by $\{\mathbb{H}_t^{\ab}\}$. We have, for $f \in L^2(d\m_{\ab})$ and $t>0$,
\begin{equation} \label{Hser}
\mathbb{H}_t^{\ab}f = \exp\Big(-t\sqrt{\mathbb{J}^{\ab}}\Big)f = \sum_{n=0}^{\infty} \exp\Big(-t\sqrt{\lambda_{\langle n \rangle }^{\ab}}\Big)
	\langle f, \Phi_n^{\ab} \rangle_{d\m_{\ab}} \Phi_n^{\ab},
\end{equation}
the convergence being in $L^2(d\m_{\ab})$. 
In fact, the last series converges pointwise for any $f \in L^p(wd\m_{\ab})$, $w \in A_p^{\ab}$, $1\le p < \infty$, and defines a smooth function of $(t,\theta) \in (0,\infty)\times (-\pi,\pi)$. Thus the series \eqref{Hser} can be regarded as a definition of $\{\mathbb{H}_t^{\ab}\}$ on the weighted spaces $L^p(wd\m_{\ab})$, $w\in A_p^{\ab}$, $1\le p < \infty$.

To give a brief justification of the pointwise convergence note that the normalized Jacobi polynomials
satisfy the estimate (see \cite[(7.32.2)]{Sz})
\begin{equation} \label{estjac}
|\P_n^{\ab}(\theta)| \lesssim (n+1)^{\alpha+\beta+2}, \qquad \theta \in (-\pi,\pi), \quad n \ge 0.
\end{equation}
Moreover, we have the differentiation rule (cf. \cite[(4.21.7)]{Sz})
\begin{equation} \label{jacdiff}
\frac{d}{d\theta} \P_n^{\ab}(\theta) = -\frac{1}{2}\sqrt{n(n+\alpha+\beta+1)}\, \sin\theta \;
	\P_{n-1}^{\alpha+1,\beta+1}(\theta), \qquad n \ge 0;
\end{equation}
here we use the convention that $\P^{\ab}_{k}\equiv 0$ if $k<0$.
From \eqref{estjac} and \eqref{jacdiff} it follows that 
\begin{equation} \label{jacdiffest}
\Big|\frac{d}{d\theta} \P_n^{\ab}(\theta)\Big|\lesssim n^{\alpha+\beta+5}, \qquad \theta \in (-\pi,\pi), \quad n \ge 0.
\end{equation}
Now combining \eqref{Phi} with \eqref{estjac} and \eqref{jacdiffest} we get
\begin{equation} \label{estsymm}
|\Phi_n^{\ab}(\theta)| \lesssim (n+1)^{\alpha+\beta+4}, \qquad \theta \in (-\pi,\pi), \quad n \ge 0.
\end{equation}
Further, using H\"older's inequality, we see that the Fourier-Jacobi coefficients of any 
$f\in L^p(wd\m_{\ab})$, $w\in A_p^{\ab}$, $1\le p < \infty$, grow at most polynomially, in the sense that
\begin{equation} \label{growthFJnew}
\big|\langle f, \Phi_n^{\ab}\rangle_{d\m_{\ab}}\big| \lesssim \|f\|_{L^p(wd\m_{\ab})} 
	(n+1)^{\alpha + \beta + 4}, \qquad n \ge 0.
\end{equation}
Therefore the series in \eqref{Hser} converges absolutely and uniformly on compact subsets of $(0,\infty)\times (-\pi,\pi)$, thanks to
the exponentially decreasing factor in $n$. Moreover, term by term differentiation of this series, with the aid of  \eqref{Phi}, \eqref{estjac} and \eqref{jacdiff}, shows that it defines a smooth function of $(t,\theta)\in (0,\infty)\times (-\pi,\pi)$.

The integral representation of $\{\mathbb{H}_t^{\ab}\}$, valid on the weighted $L^p$ spaces mentioned above
(see for instance \cite{No} for the relevant arguments), is
$$
\mathbb{H}_t^{\ab}f(\theta) = \int_{-\pi}^{\pi} \mathbb{H}_t^{\ab}(\theta,\varphi)f(\varphi)\, d\m_{\ab}(\varphi),
	\qquad \theta \in (-\pi,\pi), \quad t>0,
$$
with the symmetrized Jacobi-Poisson kernel
\begin{equation}
\mathbb{H}_t^{\ab}(\theta,\varphi) =\sum_{n=0}^{\infty} \exp\Big(-t\sqrt{\lambda_{\langle n \rangle }^{\ab}}\Big)
	\Phi_n^{\ab}(\theta)\Phi_n^{\ab}(\varphi).
\end{equation}
The last series converges absolutely and defines a smooth function of $(t,\theta,\varphi)\in (0,\infty)\times(-\pi,\pi)^2$,
as can be seen by \eqref{Phi}, \eqref{estjac}, \eqref{jacdiff} and term by term differentiation. 

We now give strict definitions in $L^2(d\m_{\ab})$ of our main objects of interest. For $f\in L^2(d\m_{\ab})$ we define
\begin{itemize}
\item[(i)] symmetrized Riesz-Jacobi transforms of order $N$
$$
\mathbb{R}_N^{\ab}f = \sum_{n=0}^{\infty} \big(\lambda_{\langle n \rangle}^{\ab}\big)^{-N/2}	\langle f, \Phi_n^{\ab} \rangle_{d\m_{\ab}} D^N \Phi_n^{\ab}, 
$$
where $N=1,2,\ldots$;

\item[(ii)] the symmetrized Jacobi-Poisson semigroup maximal operator
$$
\mathbb{H}^{\ab}_* f(\theta) = \big\|\mathbb{H}_t^{\ab}f(\theta)\big\|_{L^{\infty}(dt)}, \qquad \theta \in (-\pi,\pi);
$$

\item[(iii)]  symmetrized mixed square functions of arbitrary orders $M,N$
$$
\mathbb{G}_{M,N}^{\ab}(f)(\theta) = \big\|\partial_t^M D^N \mathbb{H}_t^{\ab}f(\theta)\big\|_{L^2(t^{2M+2N-1}dt)},\qquad \theta\in(-\pi,\pi),
$$
where $M,N=0,1,2,\ldots$ and $M+N>0$;

\item[(iv)] multipliers of Laplace and Laplace-Stieltjes transform type
$$\mathbb{M}^{\ab}f(\theta)=\sum_{n=0}^{\infty} m\Big(\sqrt{\lambda_{\langle n \rangle}^{\ab}}\Big)	\langle f, \Phi_n^{\ab} \rangle_{d\m_{\ab}}\Phi_n^{\ab},$$
where $m=m_{\phi}$ or $m=m_{\nu}$, with $m_{\phi}$ and $m_{\nu}$ as in Definition \ref{def}.

\end{itemize}

The operators $\mathbb{R}^{\ab}_N$ and $\mathbb{M}^{\ab}$ are indeed well defined on $L^2(d\m_{\ab})$ by the above formulas, since the series converge in $L^2(d\m_{\ab})$ and, moreover, the operators are bounded on $L^2(d\m_{\ab})$. This is clear in case of $\mathbb{M}^{\ab}$, because of the boundedness of $m$ on $\spectrum\sqrt{\mathbb{J}^{\ab}}$. The case of the Riesz transforms is covered by \cite[Proposition 4.3]{Symmetrized}. Note that if $\alpha+\beta=-1$ then $0=\lambda_0^{\ab}$ is the eigenvalue of $\mathbb{J}^{\ab}$ and we actually consider (cf. \cite[(4.3)]{Symmetrized}) the operators
\begin{align*}
\mathbb{R}_N^{\ab} \Pi_0 f & = \sum_{n=1}^{\infty} \big(\lambda_{\langle n \rangle}^{\ab}\big)^{-N/2}	\langle f, \Phi_n^{\ab} \rangle_{d\m_{\ab}} D^N \Phi_n^{\ab}, \qquad N=1,2,\ldots ,
\end{align*}
where $\Pi_0$ is the orthogonal projection onto $\{\Phi_0^{\ab}\}^{\perp}$. However, since $D\Phi_0^{\ab}$ vanishes, we do not distinguish between $\mathbb{R}_N^{\ab}$ and $\mathbb{R}_N^{\ab}\Pi_0$ and always denote the Riesz operators by $\mathbb{R}_N^{\ab}$. As for the remaining operators $\mathbb{H}_{*}^{\ab}$ and $\mathbb{G}_{M,N}^{\ab},$ their definitions are understood pointwise and make sense for general $f\in L^p(w d\mu_{\ab}), w \in A_p^{\ab}, 1\le p< \infty,$ since $\mathbb{H}_t^{\ab}f(\theta)$ is a smooth function of $(t,\theta)\in (0,\infty)\times(-\pi,\pi).$

The main result of the paper reads as follows.
\begin{thm} \label{thm:main}
Let $\alpha, \beta\ge-1/2$ and $w$ be an even weight on $(-\pi,\pi)$. Then the maximal operator $\mathbb{H}_{*}^{\ab}$ and the square functions $\mathbb{G}_{M,N}^{\ab}$, $M,N=0,1,2,\ldots$, $M+N>0$, are bounded on $L^p(w d\m_{\ab})$, $w\in A_p^{\ab}$, $1<p<\infty$ and from $L^1(w d\m_{\ab})$ to weak $L^1(w d\m_{\ab})$, $w \in A_1^{\ab}$.
Furthermore, the Riesz transforms $\mathbb{R}_N^{\ab}$, $N=1,2,\ldots$ and the multipliers  $\mathbb{M}^{\ab}$ extend uniquely to bounded linear operators on $L^p(w d\m_{\ab})$, $w\in A_p^{\ab}$, $1<p<\infty$ and from $L^1(w d\m_{\ab})$ to weak $L^1(w d\m_{\ab})$, $w \in A_1^{\ab}$.
\end{thm}

The task of proving the above theorem can be reduced to showing similar mapping properties for suitably defined `restricted' operators related to the smaller measure space $((0,\pi),d\mu_{\ab}^{+})$. The details are as follows. Observe that, as a consequence of \eqref{Phi} and \eqref{jacdiff}, we have the decomposition
\begin{align}\nonumber 
\mathbb{H}_t^{\ab}(\theta,\varphi)&=\frac{1}{2}\sum_{n=0}^{\infty} \exp\Big(-t\sqrt{\lambda_{n}^{\ab}}\Big)\P_{n}^{\alpha,\beta}(\theta)\P_{n}^{\alpha,\beta}(\varphi)\\
&\quad+\frac{1}{8}\sin\theta\,\sin\varphi\sum_{n=0}^{\infty} \exp\Big(-t\sqrt{\lambda_{n+1}^{\ab}}\Big)\P_{n}^{\alpha+1,\beta+1}(\theta)\P_{n}^{\alpha+1,\beta+1}(\varphi)\nonumber\\
&\equiv H_{t}^{\ab}(\theta,\varphi)+\widetilde{H}_{t}^{\ab}(\theta,\varphi).\label{deco}
\end{align}
The restriction of $H_{t}^{\ab}(\theta,\varphi)$ to $\theta,\varphi\in (0,\pi)$ coincides, up to the factor $1/2$, with the standard (non-symmetrized) Jacobi-Poisson kernel considered in \cite{NoSjogren}. Furthermore, since each $\P_{n}^{\alpha,\beta}$ is an even function on $(-\pi,\pi),$ we see that $H_{t}^{\ab}(\theta,\varphi)$ and $\widetilde{H}_{t}^{\ab}(\theta,\varphi)$ are even and odd, respectively, functions both of $\theta$ and $\varphi$. For further reference, notice also that, since $\lambda_{n+1}^{\ab}=\lambda_{n}^{\alpha+1,\beta+1}$, we have 
\begin{equation}\label{jadro}
\widetilde{H}_{t}^{\ab}(\theta,\varphi)=\frac{1}{4}\sin\theta\,\sin\varphi\,H_{t}^{\alpha+1,\beta+1}(\theta,\varphi).
\end{equation}

Next, we consider the operators acting on $L^{2}(d\mu_{\ab}^{+})$ and defined by
\begin{align*}
(H_t^{\ab})^{+}f &= \sum_{n=0}^{\infty} \exp\Big(-t\sqrt{\lambda_{n}^{\ab}}\Big)\langle f, \Phi_{2n}^{\ab} \rangle_{d\m_{\ab}^{+}} \Phi_{2n}^{\ab},\qquad f\in L^{2}(d\mu_{\ab}^{+}),\\
(\widetilde{H}_t^{\ab})^{+}f &= \sum_{n=0}^{\infty} \exp\Big(-t\sqrt{\lambda_{n+1}^{\ab}}\Big)\langle f, \Phi_{2n+1}^{\ab} \rangle_{d\m_{\ab}^{+}} \Phi_{2n+1}^{\ab},\qquad f\in L^{2}(d\mu_{\ab}^{+}),
\end{align*}
for $t>0.$ The integral representations of $(H_t^{\ab})^{+}$ and $(\widetilde{H}_t^{\ab})^{+}$ are
\begin{align*}
(H_t^{\ab})^{+}f(\theta) &=\int_0^{\pi} H_t^{\ab}(\theta,\varphi)f(\varphi)\, d\m_{\ab}^{+}(\varphi),
	\qquad \theta \in (0,\pi), \quad t>0 ,\\
(\widetilde{H}_t^{\ab})^{+}f(\theta) &=\int_0^{\pi} \widetilde{H}_t^{\ab}(\theta,\varphi)f(\varphi)\, d\m_{\ab}^{+}(\varphi),
	\qquad \theta \in (0,\pi), \quad t>0.
\end{align*}
Similarly as in the case of $\mathbb{H}_{t}^{\ab}$, we conclude that for any $f\in L^{p}(wd\mu_{\ab}^{+})$, $w\in (A_{p}^{\ab})^{+}$, $1\le p< \infty,$ the series defining $(H_t^{\ab})^{+}$ and $(\widetilde{H}_t^{\ab})^{+}$ converge pointwise and give rise to smooth functions of $(t,\theta)\in (0,\infty)\times(0,\pi).$

For $N=1,2,\dots ,$ denote 
$$
\de=\underbrace{\ldots \delta \delta^* \delta \delta^* \delta}_{N\; \textrm{components}},\qquad \ko=\underbrace{\ldots \delta^* \delta \delta^* \delta \delta^*}_{N\; \textrm{components}}.
$$ 
These derivatives correspond to the action of $D^N$ on even and odd functions, respectively. Now we define the `restricted' operators we shall investigate:
\begin{itemize}
\item[(i)]
\begin{align*}
(R_N^{\ab})^{+}f&=\sum_{n=0}^{\infty} (\lambda_{n}^{\ab})^{-N/2}	\langle f, \Phi_{2n}^{\ab} \rangle_{d\m_{\ab}^{+}} \de \Phi_{2n}^{\ab},\qquad f\in L^{2}(d\mu_{\ab}^{+}),\\
(\widetilde{R}_N^{\ab})^{+}f&=\sum_{n=0}^{\infty} (\lambda_{n+1}^{\ab})^{-N/2}	\langle f, \Phi_{2n+1}^{\ab} \rangle_{d\m_{\ab}^{+}} \ko \Phi_{2n+1}^{\ab},\qquad f\in L^{2}(d\mu_{\ab}^{+}),
\end{align*}

\item[(ii)]
\begin{align*}
({H}^{\ab}_*)^{+} f(\theta) &=\big\|(H_t^{\ab})^{+}f(\theta)\big\|_{L^{\infty}(dt)},\\
(\widetilde{H}^{\ab}_{*})^{+} f(\theta) &=\big\|(\widetilde{H}_t^{\ab})^{+}f(\theta)\big\|_{L^{\infty}(dt)},
\end{align*}

\item[(iii)]
\begin{align*}
({G_{M,N}^{\ab}})^{+}(f)(\theta) &= \big\|\partial_t^M \de (H_t^{\ab})^{+}f(\theta)\big\|_{L^2(t^{2M+2N-1}dt)},\\
(\widetilde{G}_{M,N}^{\ab})^{+}(f)(\theta) &= \big\|\partial_t^M \ko (\widetilde{H}_t^{\ab})^{+}f(\theta)\big\|_{L^2(t^{2M+2N-1}dt)},
\end{align*}

\item[(iv)]
\begin{align*}
(M^{\ab})^{+}f&=\sum_{n=0}^{\infty} m\Big(\sqrt{\lambda_{n}^{\ab}}\Big)	\langle f, \Phi_{2n}^{\ab} \rangle_{d\m_{\ab}^{+}}\Phi_{2n}^{\ab},\qquad f\in L^{2}(d\mu_{\ab}^{+}),\\
(\widetilde{M}^{\ab})^{+}f&=\sum_{n=0}^{\infty} {m}\Big(\sqrt{\lambda_{n+1}^{\ab}}\Big)	\langle f, \Phi_{2n+1}^{\ab} \rangle_{d\m_{\ab}^{+}}\Phi_{2n+1}^{\ab},\qquad f\in L^{2}(d\mu_{\ab}^{+}).
\end{align*}

\end{itemize}

We are now in a position to reduce the proof of Theorem \ref{thm:main} to showing the following statement. 

\begin{thm} \label{thm:main'}

Let $\alpha, \beta\ge-1/2$. Then the operators $({H}_{*}^{\ab})^{+}$, $(\widetilde{H}_{*}^{\ab})^{+}$, $(G_{M,N}^{\ab})^{+}$, $(\widetilde{G}_{M,N}^{\ab})^{+}$, $M,N=0,1,2,\ldots$, $M+N>0$, are bounded on $L^p(w d\m_{\ab}^{+})$, $w \in (A_p^{\ab})^{+}$, $1<p<\infty$, and from $L^1(w d\m_{\ab}^{+})$ 
to weak $L^1(w d\m_{\ab}^{+})$, $w \in (A_1^{\ab})^{+}$.
Furthermore, the operators $({R}_N^{\ab})^{+}$, $(\widetilde{R}_N^{\ab})^{+}$ $N=1,2,\ldots$, $(M^{\ab})^{+}$ and $(\widetilde{M}^{\ab})^{+}$ extend uniquely to bounded linear operators on $L^p(w d\m_{\ab}^{+})$, $w\in (A_p^{\ab})^{+}$, $1<p<\infty$, and from $L^1(w d\m_{\ab}^{+})$ 
to weak $L^1(w d\m_{\ab}^{+})$, $w\in (A_1^{\ab})^{+}$.
\end{thm}

For the sake of brevity, we give a detailed description of the reduction only in the case of $\mathbb{R}_{N}^{\ab},$ adapting suitably the arguments from the proof of
\cite[Theorem 1]{imag}. The remaining cases are treated in a similar manner and are left to the reader.

Let $1\le p< \infty$ be fixed and let $w$ be an even weight on $(-\pi,\pi)$ such that $w^{+}\in (A_p^{\ab})^{+}$ or, equivalently, $w\in A_p^{\ab}.$  First, we decompose $\mathbb{R}_{N}^{\ab}$ by splitting the summation in the defining series over even and odd $n$, getting
$$\mathbb{R}_{N}^{\ab}=R_{N}^{\ab}+\widetilde{R}_{N}^{\ab},$$
where 
\begin{align*}R_{N}^{\ab}f&=\sum_{n=0}^{\infty}  (\lambda_{n}^{\ab})^{-N/2}
	\langle f, \Phi_{2n}^{\ab} \rangle_{d\m_{\ab}} D^{N}\Phi_{2n}^{\ab},\\
	\widetilde{R}_{N}^{\ab}f&=\sum_{n=0}^{\infty}  (\lambda_{n+1}^{\ab})^{-N/2}
	\langle f, \Phi_{2n+1}^{\ab} \rangle_{d\m_{\ab}} D^{N}\Phi_{2n+1}^{\ab}.
	\end{align*}

Next, take an arbitrary function $f\in L^{2}(d\mu_{\ab})$ and split it into its even and odd parts,
$f=f_{\textrm{even}}+f_{\textrm{odd}}$. Notice, that for $N,n\ge 0,$ $D^{N}\Phi_{n}^{\ab}$ is an even (odd) function if and only if $n+N$ is even (odd). Consequently,
\begin{align*}
\mathbb{R}_{N}^{\ab}f=R_{N}^{\ab}f_{\textrm{even}}+\widetilde{R}_{N}^{\ab}f_{\textrm{odd}},\qquad f\in L^{2}(d\mu_{\ab}),
\end{align*}
and for $N$ even (odd), $R_{N}^{\ab}f_{\textrm{even}}$ is an even (odd) function and $\widetilde{R}_{N}^{\ab}f_{\textrm{odd}}$ is odd (even). Furthermore, 
 \begin{align*}
 \langle f_{\textrm{even}}, \Phi_{2n}^{\ab} \rangle_{d\m_{\ab}}&=2\langle f_{\textrm{even}}^{+}, \Phi_{2n}^{\ab} \rangle_{d\m_{\ab}^{+}},\\
 \langle f_{\textrm{odd}}, \Phi_{2n+1}^{\ab} \rangle_{d\m_{\ab}}&=2\langle f_{\textrm{odd}}^{+}, \Phi_{2n+1}^{\ab} \rangle_{d\m_{\ab}^{+}}.
 \end{align*}
 The above observations, together with the identities 
 \begin{equation*}
 (D^{N}f_{\textrm{even}})^{+}=\de f_{\textrm{even}}^{+},\qquad (D^{N}f_{\textrm{odd}})^{+}=\ko f_{\textrm{odd}}^{+},
 \end{equation*}
 imply for $f\in L^p(wd\mu_{\ab})\cap L^2(d\mu_{\ab})$ the following estimates: 
 \begin{align*}
 \|\mathbb{R}_{N}^{\ab}f\|_{L^{p}(wd\mu_{\ab})}&\le \|R_{N}^{\ab}f_{\textrm{even}}\|_{L^{p}(wd\mu_{\ab})}+ \|\widetilde{R}_{N}^{\ab}f_{\textrm{odd}}\|_{L^{p}(wd\mu_{\ab})}\\
 &\le 2^{1/p}\Big(\|R_{N}^{\ab}f_{\textrm{even}}\|_{L^{p}(w^{+}d\mu_{\ab}^{+})}+ \|\widetilde{R}_{N}^{\ab}f_{\textrm{odd}}\|_{L^{p}(w^{+}d\mu_{\ab}^{+})}\Big)\\
&= 2^{1+1/p}\Big(\|(R_{N}^{\ab})^{+}f_{\textrm{even}}^{+}\|_{L^{p}(w^{+}d\mu_{\ab}^{+})}+\|(\widetilde{R}_{N}^{\ab})^{+}f_{\textrm{odd}}^{+}\|_{L^{p}(w^{+}d\mu_{\ab}^{+})}\Big). 
 \end{align*}
 
 Finally, taking into account the relation 
 $$\|f\|_{L^{p}(wd\mu_{\ab})}\simeq \|f_{\textrm{even}}^{+}\|_{L^{p}(w^{+}d\mu_{\ab}^{+})}+\|f_{\textrm{odd}}^{+}\|_{L^{p}(w^{+}d\mu_{\ab}^{+})},$$
 we conclude that the estimates
 \begin{align*}
 \|(R_{N}^{\ab})^{+}f_{\textrm{even}}^{+}\|_{L^{p}(w^{+}d\mu_{\ab}^{+})}&\lesssim \|f_{\textrm{even}}^{+}\|_{L^{p}(w^{+}d\mu_{\ab}^{+})},\\
 \|(\widetilde{R}_{N}^{\ab})^{+}f_{\textrm{odd}}^{+}\|_{L^{p}(w^{+}d\mu_{\ab}^{+})}&\lesssim \|f_{\textrm{odd}}^{+}\|_{L^{p}(w^{+}d\mu_{\ab}^{+})}
 \end{align*}
 imply the estimate 
 $$\|\mathbb{R}_{N}^{\ab}f\|_{L^{p}(wd\mu_{\ab})}\lesssim \|f\|_{L^{p}(wd\mu_{\ab})}.$$
 As verified in a similar manner, an analogous implication involving weighted weak type $(1,1)$ inequalities is also valid. 
 
 Thus we reduced the proof of Theorem \ref{thm:main} to showing Theorem \ref{thm:main'}. The proof of the latter result is based on the general Calder\'on-Zygmund theory. Obviously, the operators $(H_*^{\ab})^{+}$, $(\widetilde{H}_*^{\ab})^{+}$ and $(G_{M,N}^{\ab})^{+}$, $(\widetilde{G}_{M,N}^{\ab})^{+}$ are not linear.
However, by the well-known trick they can be interpreted as vector-valued linear operators taking values in some Banach spaces
$\mathbb{B}$. Indeed, it is convenient to identify each of them with a linear operator which maps
scalar-valued functions of $\theta \in (0,\pi)$ to $\mathbb{B}$-valued functions of $\theta$.
The corresponding nonlinear operators in question are then obtained by taking the $\mathbb{B}$ norm
at each point $\theta$, or rather at a.a{.}\,$\theta$. Clearly, $\mathbb{B}$ will be $L^2(t^{2M+2N-1}dt)$ in the cases of $(G_{M,N}^{\ab})^{+}$ and $(\widetilde{G}_{M,N}^{\ab})^{+}$.
For $(H_{*}^{\ab})^{+}$ and $(\widetilde{H}_{*}^{\ab})^{+}$ we shall, for technical reasons, choose $\mathbb{B}$ not as $L^{\infty}(dt)$ but
as the closed and separable subspace $\mathbb{X}\subset L^{\infty}(dt)$ consisting of all continuous
functions in $(0,\infty)$ which have finite limits as $t\to 0^+$ and as $t \to \infty$.
In both cases, we shall say that the operators are \emph{associated} with the corresponding
Banach spaces $\mathbb{B}$. Similarly, the linear operators $(R_N^{\ab})^{+}$, $(\widetilde{R}_N^{\ab})^{+}$, $(M^{\ab})^{+}$ and $(\widetilde{M}^{\ab})^{+}$ will
be said to be associated with the Banach space $\mathbb{B}=\mathbb{C}$.

To obtain the boundedness results for the relevant operators, we shall show that they are either scalar- or vector-valued Calder\'on-Zygmund operators, in the sense that we now define. As always, this definition goes via the kernel. So let $\mathbb{B}$ be a Banach space and let $K(\theta,\varphi)$ be a kernel defined on
$(0,\pi)\times (0,\pi) \backslash \{(\theta,\varphi):\theta=\varphi\}$ and taking values in $\mathbb{B}$.
We say that $K(\theta,\varphi)$ is a standard kernel in the sense of the space of homogeneous type
$((0,\pi),d\m_{\ab}^{+},|\cdot|)$ if it satisfies the growth estimate
\begin{equation} \label{gr}
\|K(\theta,\varphi)\|_{\mathbb{B}} \lesssim \frac{1}{\m_{\ab}^{+}(B(\theta,|\theta-\varphi|))}
\end{equation}
and the smoothness estimates
\begin{align}
\| K(\theta,\varphi) - K(\theta',\varphi)\|_{\mathbb{B}} 
	& \lesssim \frac{|\theta-\theta'|}{|\theta-\varphi|}\;
	\frac{1}{\m_{\ab}^{+}(B(\theta,|\theta-\varphi|))}, \qquad |\theta-\varphi| > 2|\theta-\theta'|,
	\label{sm1} \\
\| K(\theta,\varphi) - K(\theta,\varphi')\|_{\mathbb{B}} & 
	\lesssim \frac{|\varphi-\varphi'|}{|\theta-\varphi|}\;
	\frac{1}{\m_{\ab}^{+}(B(\theta,|\theta-\varphi|))}, \qquad |\theta-\varphi| > 2|\varphi-\varphi'|; \label{sm2}
\end{align}
recall that $B(\theta,r)$ denotes the ball (interval) centered at $\theta$ and of radius $r$.
When $K(\theta,\varphi)$ is scalar-valued, i.e. $\mathbb{B}=\mathbb{C}$, the difference conditions
\eqref{sm1} and \eqref{sm2} can be replaced by the more convenient gradient condition
\begin{equation} \label{sm}
|\partial_{\theta} K(\theta,\varphi)| + |\partial_{\varphi} K(\theta,\varphi)| \lesssim
\frac{1}{|\theta-\varphi| \m_{\ab}^{+}(B(\theta,|\theta-\varphi|))}.
\end{equation}

A linear operator $T$ assigning to each $f \in L^2(d\m_{\ab}^{+})$ a measurable $\mathbb{B}$-valued
function $Tf$ on $(0,\pi)$ is said to be a (vector-valued) Calder\'on-Zygmund operator in the sense of 
the space $((0,\pi),d\m_{\ab}^{+},|\cdot|)$ associated with $\mathbb{B}$ if
\begin{itemize}
\item[(a)] $T$ is bounded from $L^2(d\m_{\ab}^{+})$ to $L^2_{\mathbb{B}}(d\m_{\ab}^{+})$, and
\item[(b)] there exists a standard $\mathbb{B}$-valued kernel $K(\theta,\varphi)$ such that
\begin{equation*}
Tf(\theta) = \int_{0}^{\pi} K(\theta,\varphi) f(\varphi) \, d\m_{\ab}^{+}(\varphi), \qquad \textrm{a.e.}
	\;\; \theta \notin \support f,
\end{equation*}
for every $f\in L^2(d\m_{\ab}^{+})$ with compact support in $(0,\pi)$.
\end{itemize}
When $(b)$ holds, we write $T \sim K(\theta,\varphi)$ and say that $T$ is associated with $K$.
Here integration of $\mathbb{B}$-valued functions is understood in Bochner's sense, and 
$L^2_{\mathbb{B}}(d\m_{\ab}^{+})$ is the Bochner-Lebesgue space of all $\mathbb{B}$-valued $d\m_{\ab}^{+}$-square
integrable functions on $(0,\pi)$. It is well known that a large part of the classical theory
of Calder\'on-Zygmund operators remains valid, with appropriate adjustments, when the underlying
space is of homogeneous type and the associated kernels are vector-valued, see for instance the comments
in \cite[p.\,649]{NS} and references given there.

The following result, combined with the general theory of Calder\'on-Zygmund operators and standard arguments (see the proofs of \cite[Theorem 2.1]{NS} and \cite[Proposition 2.5]{szarek}) implies Theorem \ref{thm:main'}, and thus also Theorem \ref{thm:main} by the reduction reasoning described above.
 
\begin{thm} \label{thm:main2}
Assume that $\alpha, \beta \ge -1\slash 2$. The operators $(R_N^{\ab})^{+}$, $(\widetilde{R}_N^{\ab})^{+}$, $N=1,2,\ldots$, $(M^{\ab})^{+}$ and $(\widetilde{M}^{\ab})^{+}$ are Calder\'on-Zygmund operators in the sense of 
the space of homogeneous type $((0,\pi),d\m_{\ab}^{+},|\cdot|)$. 
Further, each of the operators $(H_{*}^{\ab})^{+}$, $(\widetilde{H}_{*}^{\ab})^{+}$, $(G_{M,N}^{\ab})^{+}$ and $(\widetilde{G}_{M,N}^{\ab})^{+}$, $M,N=0,1,2,\ldots$, $M+N>0$, viewed as a vector-valued operator, is a Calder\'on-Zygmund operator 
in the sense of the space $((0,\pi),d\m_{\ab}^{+},|\cdot|)$ associated with $\mathbb{B}$, and here 
$\mathbb{B}$ is $\mathbb{X}$, in case of the maximal operators, and $\mathbb{B}=L^2(t^{2M+2N-1}dt)$ for the square functions.
\end{thm}
 
The proof of Theorem \ref{thm:main2} splits naturally into showing the following three results.

\begin{prop} \label{prop:L2}
Let $\alpha,\beta \ge -1\slash 2$. The operators $(R_N^{\ab})^{+}$, $(\widetilde{R}_N^{\ab})^{+}$, $N=1,2,\ldots$, $(M^{\ab})^{+}$, $(\widetilde{M}^{\ab})^{+}$, $(H_{*}^{\ab})^{+}$,$(\widetilde{H}_{*}^{\ab})^{+}$, $(G_{M,N}^{\ab})^{+}$ and $(\widetilde{G}_{M,N}^{\ab})^{+}$,
$M,N=0,1,2,\ldots$, $M+N>0$, are bounded on $L^2(d\m_{\ab}^{+})$.
In particular, each of the operators $(H_{*}^{\ab})^{+}$,$(\widetilde{H}_{*}^{\ab})^{+}$, $(G_{M,N}^{\ab})^{+}$ and $(\widetilde{G}_{M,N}^{\ab})^{+}$,
$M,N=0,1,2,\ldots$, $M+N>0$, viewed as a vector-valued operator,
is bounded from $L^2(d\m_{\ab}^{+})$ to $L^2_{\mathbb{B}}(d\m_{\ab}^{+})$, where $\mathbb{B}$ 
is as in Theorem~\ref{thm:main2}.
\end{prop}

For $\alpha,\beta \ge -1\slash 2$ and $\theta, \varphi\in(0,\pi)$ define the kernels

\begin{align*}
M^{\ab}_{\phi}(\theta,\varphi) & = -\int_0^{\infty}
	\partial_{t}H_t^{\ab}(\theta,\varphi) \phi(t)\,dt,  \\
\widetilde{M}^{\ab}_{\phi}(\theta,\varphi) & = -\int_0^{\infty}
	\partial_{t}\widetilde{H}_t^{\ab}(\theta,\varphi) \phi(t)\,dt,  \\
M^{\ab}_{\nu}(\theta,\varphi) & = \int_{(0,\infty)} H_t^{\ab}(\theta,\varphi)\,d\nu(t), \\
\widetilde{M}^{\ab}_{\nu}(\theta,\varphi) & = \int_{(0,\infty)}
	\widetilde{H}_t^{\ab}(\theta,\varphi)\,d\nu(t),  \\	
R_N^{\ab}(\theta,\varphi) & = \frac{1}{\Gamma(N)}\int_0^{\infty} 
	\de H_t^{\ab}(\theta,\varphi) t^{N-1}\,dt, \qquad N\ge 1,\\ 
\widetilde{R}_N^{\ab}(\theta,\varphi) & = \frac{1}{\Gamma(N)}\int_0^{\infty} 
	\ko \widetilde{H}_t^{\ab}(\theta,\varphi) t^{N-1}\,dt, \qquad N\ge 1. 
\end{align*}
Here and elsewhere the derivatives $\de$ and $\ko$ act always on $\theta$ variable.

\begin{prop} \label{prop:assoc}
Let $\alpha,\beta \ge -1\slash 2$. The operators $(R_N^{\ab})^{+}$, $(\widetilde{R}_N^{\ab})^{+}$, $N=1,2,\ldots$, $(M^{\ab})^{+}$ and $(\widetilde{M}^{\ab})^{+}$ are associated with the following scalar-valued kernels:
$$
\begin{array}{lll}
&(R_N^{\ab})^{+} \sim R_N^{\ab}(\theta,\varphi), \qquad 
&(\widetilde{R}_N^{\ab})^{+} \sim \widetilde{R}_N^{\ab}(\theta,\varphi),\\
&({M}^{\ab})^{+} \sim {M}^{\ab}(\theta,\varphi), \qquad 
&(\widetilde{M}^{\ab})^{+} \sim \widetilde{M}^{\ab}(\theta,\varphi), 
\end{array}
$$
where ${M}^{\ab}$ is equal either to $M^{\ab}_{\phi}$ or $M^{\ab}_{\nu}$, depending on whether $m=m_{\phi}$ or $m=m_{\nu}$, respectively, and similarly for the corresponding objects with tildes.

Further, the operators $(H_{*}^{\ab})^{+}$,$(\widetilde{H}_{*}^{\ab})^{+}$, $(G_{M,N}^{\ab})^{+}$, $(\widetilde{G}_{M,N}^{\ab})^{+},$
$M,N=0,1,2,\ldots$, $M+N>0$, viewed as vector-valued operators,
are associated with the following $\mathbb{B}$-valued kernels:
$$
\begin{array}{lll}
& (H_{*}^{\ab})^{+}   \sim \{H_t^{\ab}(\theta,\varphi)\}_{t>0}, \qquad 
& (\widetilde{H}_{*}^{\ab})^{+}   \sim \{\widetilde{H}_t^{\ab}(\theta,\varphi)\}_{t>0},\\
& (G_{M,N}^{\ab})^{+}  \sim \{\partial_t^{M}\delta_N^{\emph{\textrm{even}}} H_t^{\ab}(\theta,\varphi)\}_{t>0}, \qquad
& (\widetilde{G}_{M,N}^{\ab})^{+}  \sim \{\partial_t^{M}\delta_N^{\emph{\textrm{odd}}} \widetilde{H}_t^{\ab}(\theta,\varphi)\}_{t>0}.
\end{array}
$$
Here $\mathbb{B}$ is as in Theorem \ref{thm:main2}.
\end{prop}

\begin{thm} \label{thm:stand}
Assume that $\alpha,\beta \ge -1\slash 2$. The scalar-valued kernels from Proposition \ref{prop:assoc}
satisfy the standard estimates \eqref{gr}, with $\mathbb{B}=\mathbb{C}$, and \eqref{sm}.
The vector-valued kernels appearing in Proposition \ref{prop:assoc} satisfy
the standard estimates \eqref{gr}, \eqref{sm1} and \eqref{sm2}, with $\mathbb{B}$ as in Theorem \ref{thm:main2}.
\end{thm}

The proofs of Propositions \ref{prop:L2} and \ref{prop:assoc} are given in Section \ref{sec:L^2}. The proof of Theorem \ref{thm:stand} is the most technical part of the paper and is located in Section \ref{sec:ker}.

\section{$L^2$-boundedness and kernel associations} \label{sec:L^2}

 In the arguments we give in this section we will frequently use, without further mention, the fact that each of the systems $\{\sqrt{2}\Phi_{2n}^{\ab}: n\ge 0\}$ and $\{\sqrt{2}\Phi_{2n+1}^{\ab}: n\ge 0\}$, if restricted to the interval $(0,\pi)$, forms an orthonormal basis in $L^2(d\mu_{\ab}^{+}).$

We first show Proposition \ref{prop:L2}. To prove $L^2$-boundedness of the Riesz operators and the square functions, we shall need the following result that explains the way the higher order `derivatives' $D^N$ act on $\Phi_{n}^{\ab}$.

\begin{lem}[{\cite[Corollary 4.3]{Symmetrized}}] \label{cor:C11}
Given $\el \ge 1$, we have
$$
D^{\el} \Phi_n^{\ab} = (-1)^{\el\slash 2 + (n+3\slash 2)\widetilde{\el}}
	\big(\lambda^{\ab}_{\langle n \rangle}-\lambda_{0}^{\ab}\big)^{\el\slash 2}
	\Phi_{n-(-1)^n\widetilde{\el}}^{\ab}\,, \qquad n \ge 1,
$$
where $\widetilde{\el}=0$ if $\el$ is even and
$\widetilde{\el}=1$ otherwise.
\end{lem}

\begin{cor} \label{cor:delta}
Given $N\ge 1$, we have
\begin{align*}
\delta_N^{\emph{\textrm{even}}} \Phi_{2n}^{\ab} &= (-1)^{{N}\slash 2 - \widetilde{N}\slash 2}
	 \big(\lambda_n^{\ab}-\lambda_{0}^{\ab})^{N\slash 2}
	\Phi_{2n-\widetilde{N}}^{\ab}, \qquad n \ge 0,\\
\delta_N^{\emph{\textrm{odd}}} \Phi_{2n+1}^{\ab} &= (-1)^{{N}\slash 2 + \widetilde{N}\slash 2}
	 \big(\lambda_{n+1}^{\ab}-\lambda_{0}^{\ab})^{N\slash 2}
	\Phi_{2n+1+\widetilde{N}}^{\ab}, \qquad n \ge 0,
\end{align*}
where $\widetilde{N}=0$ if $N$ is even and $\widetilde{N}=1$ if N is odd, with the convention that $\Phi_k^{\ab}\equiv 0$ if $k<0$.
\end{cor}

\begin{proof}[Proof of Proposition \ref{prop:L2}]
For $(M^{\ab})^{+}$ and $(\widetilde{M}^{\ab})^{+}$ the conclusion is immediate, since both $m_{\phi}$ and $m_{\nu}$ are bounded functions on the set $\{(\lambda_{n}^{\ab})^{1/2}: n\ge 0\}$.

Next, we treat the maximal operator. Since $({H}_*^{\ab})^{+}$ coincides, up to a constant factor, with the Jacobi-Poisson semigroup maximal operator $\mathcal{H}_*^{\ab}$ investigated in \cite{NoSjogren}, its $L^2(d\mu_{\ab}^{+})$-boundedness follows immediately from the analogous result for $\mathcal{H}_*^{\ab}$ stated in \cite[Proposition 2.2]{NoSjogren}. Now observe that this property of $({H}_*^{\alpha+1,\beta+1})^{+}$ implies the same for $(\widetilde{H}_*^{\ab})^{+}.$ Indeed, for $f\in L^{2}(d\mu_{\ab}^{+})$ we have

\begin{align*}
(\widetilde{H}_*^{\ab})^{+}f(\theta)&=\underset{t>0}{\sup}\Big|\frac{\sin\theta}{4}\int_0^{\pi} H_t^{\alpha+1,\beta+1}(\theta,\varphi)f(\varphi)\sin\varphi\, d\mu_{\ab}^{+}(\varphi)\Big|\\
&=\sin\theta\, (H_*^{\alpha+1,\beta+1})^{+}\Big(\frac{f}{\sin}\Big)(\theta).
\end{align*}
Thus, using the $L^2$-boundedness of $(H_*^{\alpha+1, \beta+1})^{+}$, we get
\begin{align*}
\|(\widetilde{H}_*^{\ab})^{+}f\|_{L^2(d\mu_{\ab}^{+})}&=\bigg(\int_0^{\pi}\Big|({H}_*^{\alpha+1,\beta+1})^{+}\Big(\frac{f}{\sin}\Big)(\theta)\Big|^{2}\sin^{2}\theta\, d\mu_{\ab}^{+}(\theta)\bigg)^{1/2}\\
&=2\Big\|({H}_*^{\alpha+1,\beta+1})^{+}\Big(\frac{f}{\sin}\Big)\Big\|_{L^2(d\mu_{\alpha+1,\beta+1}^{+})}\\
&\lesssim \Big\|\frac{f}{\sin}\Big\|_{L^2(d\mu_{\alpha+1,\beta+1}^{+})}\simeq \|f\|_{L^2(d\mu_{\ab}^{+})}.
\end{align*}

We pass to the cases of $(R_{N}^{\ab})^{+}$ and $(\widetilde{R}_{N}^{\ab})^{+}$, and focus only on $(R_{N}^{\ab})^{+}$ since the argument for $(\widetilde{R}_{N}^{\ab})^{+}$ is essentially the same. Using Corollary \ref{cor:delta} we can write
\begin{equation*}
(R_{N}^{\ab})^{+}f=\sum_{n=0}^{\infty}(\lambda_{n}^{\ab})^{-N/2}\langle f,\Phi_{2n}^{\ab} \rangle_{d\m_{\ab}^{+}} (-1)^{{N}\slash 2 - \widetilde{N}\slash 2}(\lambda_{n}^{\ab}-\lambda_{0}^{\ab})^{N/2}\Phi_{2n-\widetilde{N}}^{\ab}. 
\end{equation*}
From here the conclusion follows by Parseval's identity,
\begin{align*}
\|(R_{N}^{\ab})^{+}f\|_{L^2(d\m_{\ab}^{+})}^2&=\frac{1}{2}\sum_{n=0}^{\infty}\Big(\frac{\lambda_{n}^{\ab}-\lambda_{0}^{\ab}}{\lambda_{n}^{\ab}}\Big)^N\big|\langle f,\Phi_{2n}^{\ab} \rangle_{d\m_{\ab}^{+}}\big|^2\\
&\le\frac{1}{2}\sum_{n=0}^{\infty}\big|\langle f,\Phi_{2n}^{\ab} \rangle_{d\m_{\ab}^{+}}\big|^2=\frac{1}{4}\|f\|_{L^2(d\m_{\ab}^{+})}^2.
\end{align*}
Notice that, by convention, in the exceptional case $\alpha+\beta+1=0$ there is no term with $n=0$ in the last few sums.

Finally, to prove the assertion for the $g$-functions we treat $(\widetilde{G}_{M,N}^{\ab})^{+}$ and leave the analogous case of $({G}_{M,N}^{\ab})^{+}$ to the reader. In view of Corollary \ref{cor:delta}, \eqref{estsymm} and \eqref{growthFJnew}, we have
\begin{align*}
\partial_t^M \ko (\widetilde{H}_t^{\ab})^{+}f(\theta)&=\sum_{n=0}^{\infty}(-1)^{M}(\lambda_{n+1}^{\ab})^{M/2} e^{-t(\lambda_{n+1}^{\ab})^{1/2}}
	\langle f, \Phi_{2n+1}^{\ab} \rangle_{d\m_{\ab}^{+}} \ko \Phi_{2n+1}^{\ab}(\theta)\\
&=\sum_{n=0}^{\infty}(-1)^{M}(\lambda_{n+1}^{\ab})^{M/2} e^{-t(\lambda_{n+1}^{\ab})^{1/2}}\langle f, \Phi_{2n+1}^{\ab} \rangle_{d\m_{\ab}^{+}}\\
&\quad\times(-1)^{N/2+\widetilde{N}/2}(\lambda_{n+1}^{\ab}-\lambda_{0}^{\ab})^{N/2}\Phi_{2n+1+\widetilde{N}}^{\ab}(\theta).
\end{align*}
Therefore, changing the order of integration and using Parseval's identity, we get
\begin{align*}
\|(\widetilde{G}_{M,N}^{\ab})^{+}f\|_{L^2(d\m_{\ab}^{+})}^2 &=\int_{0}^{\infty}\int_{0}^{\pi}|\partial_t^M \ko(\widetilde{H}_t^{\ab})^{+}f(\theta)|^{2}d\m_{\ab}^{+}(\theta)\,t^{2M+2N-1}dt\\
&=\frac{1}{2}\int_{0}^{\infty}\sum_{n=0}^{\infty}(\lambda_{n+1}^{\ab})^{M}(\lambda_{n+1}^{\ab}-\lambda_{0}^{\ab})^{N} e^{-2t(\lambda_{n+1}^{\ab})^{1/2}}\\
&\quad\times\big|\langle f, \Phi_{2n+1}^{\ab} \rangle_{d\m_{\ab}^{+}}\big|^{2}t^{2M+2N-1}dt.
\end{align*}
Integrating in $t$, we see that
\begin{align*}
\|(\widetilde{G}_{M,N}^{\ab})^{+}f\|_{L^2(d\m_{\ab}^{+})}^2&=\bigg(\frac{\lambda_{n+1}^{\ab}-\lambda_{0}^{\ab}}{\lambda_{n+1}^{\ab}}\bigg)^{N}\,\frac{\Gamma(2M+2N)}{2^{2M+2N+1}}\sum_{n=0}^{\infty}\big|\langle f, \Phi_{2n+1}^{\ab} \rangle_{d\m_{\ab}^{+}}\big|^{2}\\
&\le \frac{\Gamma(2M+2N)}{2^{2M+2N+1}}\sum_{n=0}^{\infty}\big|\langle f, \Phi_{2n+1}^{\ab} \rangle_{d\m_{\ab}^{+}}\big|^{2}=\frac{\Gamma(2M+2N)}{2^{2M+2N+2}}\|f\|_{L^2(d\m_{\ab}^{+})}^2.
\end{align*}

The proof of Proposition \ref{prop:L2} is finished.
\end{proof}

A general outline of the procedure to prove Proposition \ref{prop:assoc} can be found in the proof of \cite[Proposition 2.3]{NoSjogren}. In case of each of the operators in question, the treatment relies on similar arguments. Therefore, for the sake of brevity, here we give details only for the two representative cases $(R_{N}^{\ab})^{+}$ and $(\widetilde{G}_{M,N}^{\ab})^{+}.$ As for the remaining cases, the reader might like to consult the corresponding proofs in the Laguerre setting given in \cite{NS} (Riesz transforms and the maximal operator), in \cite{szarek} ($g$-functions) and in \cite{szarek2} (Laplace multipliers in a more general Dunkl setting); see also the proof of \cite[Proposition 2.3]{NoSjogren} for Laguerre-based sketches of the proofs in the Jacobi situation.

In the proof below we will make use of some kernel estimates that are proved independently in Section \ref{sec:ker}. We emphasize that the kernels we are dealing with have non-integrable singularities along the diagonal and operations like applying Fubini's theorem or changing the order of integration (summation) and differentiation are delicate matters. Therefore we provide fairly detailed explanations in the relevant places.

\begin{proof}[Proof of Proposition \ref{prop:assoc}; the case of $(R_{N}^{\ab})^{+}$]
By density arguments the proof reduces to showing that for $f,g\in C_{c}^{\infty}(0,\pi)$ such that $\support f\cap \support g=\emptyset$ we have

\begin{equation}\label{tozs}
\big\langle (R_{N}^{\ab})^{+}f,g \big\rangle_{d\mu_{\ab}^{+}}=\int_{0}^{\pi}\int_{0}^{\pi} 
R_{N}^{\ab}(\theta,\varphi)f(\varphi)\overline{g(\theta)}\,d\mu_{\ab}^{+}(\varphi)\,d\mu_{\ab}^{+}(\theta).
\end{equation}
By orthogonality of the system $\{\de \Phi_{2n}^{\ab}:n\ge 0\}$ it follows that
\begin{align*}
\big\langle (R_{N}^{\ab})^{+}f,g \big\rangle_{d\mu_{\ab}^{+}}=\sum_{n=0}^{\infty}  (\lambda_{n}^{\ab})^{-N/2}
	\langle f, \Phi_{2n}^{\ab} \rangle_{d\m_{\ab}^{+}}\overline{\langle g, \de \Phi_{2n}^{\ab} \rangle_{d\m_{\ab}^{+}}}
\end{align*}
(recall that in the critical case $\alpha+\beta=-1$ the summation here and in few places below starts actually from $n=1$). Thus it remains to show that the right-hand sides here and in \eqref{tozs} coincide. This will be done once we justify the following chain of identities:
\begin{align*}
&\Gamma(N)\int_{0}^{\pi}\int_{0}^{\pi} R_{N}^{\ab}(\theta,\varphi)f(\varphi)\overline{g(\theta)}\,d\mu_{\ab}^{+}(\varphi)\,d\mu_{\ab}^{+}(\theta)\\
&=\int_0^{\infty}\int_{0}^{\pi}\int_{0}^{\pi} \de H_t^{\ab}(\theta,\varphi) f(\varphi)\overline{g(\theta)}\,d\mu_{\ab}^{+}(\varphi)\,d\mu_{\ab}^{+}(\theta)\,t^{N-1}dt\\
&=\int_0^{\infty}\int_{0}^{\pi}\int_{0}^{\pi}\sum_{n=0}^{\infty}e^{-t(\lambda_{n}^{\ab})^{1/2}}\de \Phi_{2n}^{\ab}(\theta)\Phi_{2n}^{\ab}(\varphi) f(\varphi)\overline{g(\theta)}\,d\mu_{\ab}^{+}(\varphi)\,d\mu_{\ab}^{+}(\theta)\,t^{N-1}dt\\
&=\int_0^{\infty}\sum_{n=0}^{\infty} e^{-t(\lambda_{n}^{\ab})^{1/2}}{\langle f, \Phi_{2n}^{\ab}\rangle}_{d\m_{\ab}^{+}} \overline{\langle g,\de \Phi_{2n}^{\ab}\rangle_{d\m_{\ab}^{+}}}\,t^{N-1}dt\\
&=\Gamma(N)\sum_{n=0}^{\infty} (\lambda_{n}^{\ab})^{-N/2}{\langle f,\Phi_{2n}^{\ab}\rangle}_{d\m_{\ab}^{+}} \overline{\langle g,\de \Phi_{2n}^{\ab}\rangle_{d\m_{\ab}^{+}}}.
\end{align*}

The first identity is a consequence of Fubini's theorem. Its application is indeed legitimate since
$$\int_{0}^{\pi}\int_{0}^{\pi}\int_0^{\infty}|\de H_t^{\ab}(\theta,\varphi) f(\varphi)\overline{g(\theta)}|\,t^{N-1}dt\,d\mu_{\ab}^{+}(\varphi)d\mu_{\ab}^{+}(\theta)<\infty,$$
by the assumptions made on the supports of $f$ and $g$ and the estimate
$$\int_0^{\infty}|\de H_t^{\ab}(\theta,\varphi)| t^{N-1}dt\lesssim \frac{1}{\mu_{\ab}^{+}(B(\theta,|\theta-\varphi|))}$$
proved in Section \ref{sec:ker} while showing the growth condition for $R_{N}^{\ab}(\theta,\varphi)$. The second identity follows with the aid of \eqref{estsymm} and Corollary \ref{cor:delta}. The third equality is valid in view of Fubini's theorem, since
\begin{align*}
&\sum_{n=0}^{\infty} e^{-t(\lambda_{n}^{\ab})^{1/2}}{\langle |f|,|\Phi_{2n}^{\ab}|\rangle}_{d\m_{\ab}^{+}}{\langle |g|,|\de \Phi_{2n}^{\ab}|\rangle}_{d\m_{\ab}^{+}}<\infty,
\end{align*}
see \eqref{estsymm} and Corollary \ref{cor:delta}. Still another application of Fubini's theorem justifies the fourth identity. Its use is possible since
\begin{align*}
\sum_{n=0}^{\infty} (\lambda_{n}^{\ab})^{-N/2}\big|{\langle f,\de \Phi_{2n}^{\ab}\rangle}_{d\m_{\ab}^{+}}\big| \big|\langle g,\Phi_{2n}^{\ab}\rangle_{d\m_{\ab}^{+}}\big|<\infty,
\end{align*}
because of orthogonality of each of the systems $\{\de \Phi_{2n}^{\ab}:n\ge 0\}$ and $\{\Phi_{2n}^{\ab}:n\ge 0\}$ and Schwarz' inequality.
\end{proof}

\begin{proof}[Proof of Proposition \ref{prop:assoc}; the case of $(\widetilde{G}_{M,N}^{\ab})^{+}$] Let $\mathbb{B}=L^2(t^{2M+2N-1}dt)$. Taking into account the fact that the space $L_{\mathbb{B}}^2((\support f)^{c},d\mu_{\ab}^{+})$ is self-dual and standard density arguments we infer that it is enough to check that
\begin{align} \nonumber
&{\Big\langle \big\{\partial_{t}^{M}\ko(\widetilde{H}_t^{\ab})^{+}f\big\}_{t>0},h 
\Big\rangle}_{L_{\mathbb{B}}^2(d\mu_{\ab}^{+})}\\
&=\bigg\langle\int_{0}^{\pi}\big\{\partial_{t}^{M}\ko\widetilde{H}_t^{\ab}(\theta,\varphi)\big\}_{t>0}f(\varphi)d\mu_{\ab}^{+}(\varphi),h\bigg\rangle_{L_{\mathbb{B}}^2(d\mu_{\ab}^{+})},\label{stowg}
\end{align}
for any $f\in C_{c}^{\infty}(0,\pi)$ and $h(\theta,t)=h_{1}(\theta)h_{2}(t)$, where $h_{1}\in C_{c}^{\infty}(0,\pi)$, $h_{2}\in C_{c}^{\infty}(0,\infty)$ and $\support f \cap \support h_{1}=\emptyset$ (the linear span of functions $h$ of this form is dense in $L^2\big((\support f)^{c}\times(0,\infty),d\mu_{\ab}^{+}\otimes t^{2M+2N-1}dt\big)$).

We begin with considering the left-hand side of \eqref{stowg},
\begin{align*}
&{\big\langle \{\partial_{t}^{M}\ko(\widetilde{H}_t^{\ab})^{+}f\}_{t>0},h \big\rangle}_{L_{\mathbb{B}}^2(d\mu_{\ab}^{+})}\\
&=\int_{0}^{\infty}\int_{0}^{\pi}\partial_{t}^{M}\ko (\widetilde{H}_t^{\ab})^{+}f(\theta)\overline{h_{1}(\theta)}\,d\mu_{\ab}^{+}(\theta)\,\overline{h_{2}(t)}\,t^{2M+2N-1}dt\\
&=\int_{0}^{\infty}\int_{0}^{\pi}\sum_{n=0}^{\infty} (-1)^{M}(\lambda_{n+1}^{\ab})^{M/2}e^{-t(\lambda_{n+1}^{\ab})^{1/2}}{\langle f,\Phi_{2n+1}^{\ab}\rangle}_{d\m_{\ab}^{+}}\\
&\quad\times\ko \Phi_{2n+1}^{\ab}(\theta)\overline{h_{1}(\theta)}\,d\m_{\ab}^{+}(\theta)\,\overline{h_{2}(t)}\,t^{2M+2N-1}\,dt\\
&=\int_{0}^{\infty}\sum_{n=0}^{\infty} (-1)^{M}(\lambda_{n+1}^{\ab})^{M/2}e^{-t(\lambda_{n+1}^{\ab})^{1/2}}{\langle f,\Phi_{2n+1}^{\ab}\rangle}_{d\m_{\ab}^{+}}\overline{{\langle h_{1},\ko \Phi_{2n+1}^{\ab}\rangle}_{d\m_{\ab}^{+}}}\,\overline{h_{2}(t)}\,t^{2M+2N-1}\,dt.
\end{align*}
Here the first identity follows by Fubini's theorem. Its application is justified since
\begin{align*}
&\int_{0}^{\pi}\int_{0}^{\infty}|\partial_{t}^{M}\ko(\widetilde{H}_t^{\ab})^{+}f(\theta)\overline{h_{1}(\theta)}\overline{h_{2}(t)}|\,t^{2M+2N-1}\,dt\,d\mu_{\ab}^{+}(\theta)\\
&\le\|(\widetilde{G}_{M,N}^{\ab})^{+}(f)(\theta)\|_{L^{2}(d\mu_{\ab}^{+})}\|h_{1}\|_{L^2(d\mu_{\ab}^{+})}\,\|h_{2}\|_{L^2(t^{2M+2N-1}\,dt)}<\infty,
\end{align*}
by Schwarz' inequality and the $L^2$-boundedness of $(\widetilde{G}_{M,N}^{\ab})^{+}$, see Proposition \ref{prop:L2}. The second equality is obtained by exchanging the order of differentiation and summation, and is verified by means of \eqref{estsymm} and Corollary \ref{cor:delta}. The third identity is a consequence of Fubini's theorem, which can be applied because
$$\sum_{n=0}^{\infty} (\lambda_{n+1}^{\ab})^{M/2}e^{-t(\lambda_{n+1}^{\ab})^{1/2}}\big|{\langle f,\Phi_{2n+1}^{\ab}\rangle}_{d\m_{\ab}^{+}}\big|\,{\langle |h_{1}|,|\ko \Phi_{2n+1}^{\ab}|\rangle}_{d\m_{\ab}^{+}}<\infty,$$
see \eqref{estsymm}, \eqref{growthFJnew} and Corollary \ref{cor:delta}.

We now focus on the right-hand side of \eqref{stowg}. We have
\begin{align*}
&\bigg\langle\int_{0}^{\pi}\big\{\partial_{t}^{M}\ko\widetilde{H}_t^{\ab}(\theta,\varphi)\big\}_{t>0}f(\varphi)d\mu_{\ab}^{+}(\varphi),h\bigg\rangle_{L_{\mathbb{B}}^2(d\mu_{\ab}^{+})}\\
&=\int_{0}^{\infty}\int_{0}^{\pi}\int_{0}^{\pi}\partial_{t}^{M}\,\ko\widetilde{H}_t^{\ab}(\theta,\varphi)f(\varphi)\overline{h_{1}(\theta)}\,d\mu_{\ab}^{+}(\varphi)\,d\mu_{\ab}^{+}(\theta)\,\overline{h_{2}(t)}\,t^{2M+2N-1}\,dt\\
&=\int_{0}^{\infty}\int_{0}^{\pi}\int_{0}^{\pi}\sum_{n=0}^{\infty} (-1)^{M}(\lambda_{n+1}^{\ab})^{M/2}e^{-t(\lambda_{n+1}^{\ab})^{1/2}} \ko\Phi_{2n+1}^{\ab}(\theta)\Phi_{2n+1}^{\ab}(\varphi)\\
&\quad\times f(\varphi)\overline{h_{1}(\theta)}\,d\mu_{\ab}^{+}(\varphi)\,d\mu_{\ab}^{+}(\theta)\,\overline{h_{2}(t)}\,t^{2M+2N-1}\,dt\\
&=\int_{0}^{\infty}\sum_{n=0}^{\infty} (-1)^{M}(\lambda_{n+1}^{\ab})^{M/2}e^{-t(\lambda_{n+1}^{\ab})^{1/2}}{\langle f,\Phi_{2n+1}^{\ab}\rangle}_{d\m_{\ab}^{+}}\overline{\langle h_{1},\ko\Phi_{2n+1}^{\ab}\rangle_{d\m_{\ab}^{+}}}\,\overline{h_{2}(t)}\,t^{2M+2N-1}\,dt.
\end{align*}
The first of the above identities follows by interchanging the order of integration. Fubini's theorem can be applied since, by the growth estimate for the kernel $\{\partial_t^M\ko\widetilde{H}_{t}^{\ab}(\theta,\varphi)\}_{t>0}$ proved in Section \ref{sec:ker} and the fact that the supports of $f$ and $h_1$ are disjoint and bounded, we know that
\begin{align*}
&\int_{0}^{\pi}\int_{0}^{\pi}\int_{0}^{\infty}|\partial_{t}^{M}\ko\widetilde{H}_t^{\ab}(\theta,\varphi)f(\varphi)\overline{h_{1}(\theta)}\overline{h_{2}(t)}|t^{2M+2N-1}\,dt\,d\mu_{\ab}^{+}(\varphi)\,d\mu_{\ab}^{+}(\theta)\\
&\le\|f\|_{\infty}\,\|h_{1}\|_{\infty}\,\|h_{2}\|_{L^{2}(t^{2M+2N-1}\,dt)}\\
&\quad\times\int_{\support{h_{1}}}\int_{\support{f}}\big\|\partial_{t}^{M}\ko\widetilde{H}_t^{\ab}(\theta,\varphi)\big\|_{L^{2}(t^{2M+2N-1}\,dt)}\,d\mu_{\ab}^{+}(\varphi)d\mu_{\ab}^{+}(\theta)\\
&\lesssim\int_{\support{h_{1}}}\int_{\support{f}}\frac{1}{\mu_{\ab}^{+}(B(\theta,|\theta-\varphi|))}\,d\mu_{\ab}^{+}(\varphi)\,d\mu_{\ab}^{+}(\theta)<\infty.
\end{align*}
Entering with the derivatives under the summation in the second identity is also legitimate, as can be verified by means of Corollary \ref{cor:delta} and \eqref{estsymm}. Finally, in the third equality we interchange integration with summation, and this is justified since the estimate
\begin{align*}
&\sum_{n=0}^{\infty} (\lambda_{n+1}^{\ab})^{M/2}e^{-t(\lambda_{n+1}^{\ab})^{1/2}}{\langle |f|,|\Phi_{2n+1}^{\ab}|\rangle}_{d\m_{\ab}^{+}}{\langle |h_{1}|,|\ko\Phi_{2n+1}^{\ab}|\rangle}_{d\m_{\ab}^{+}}<\infty,
\end{align*}
see Corollary \ref{cor:delta} and \eqref{estsymm}, enables us to apply Fubini's theorem.

We conclude that both sides of \eqref{stowg} coincide. This finishes the proof.

\end{proof}

\section{Kernel estimates} \label{sec:ker}

This section is devoted to proving the kernel estimates claimed in Theorem \ref{thm:stand}. All kernels we need to estimate are expressible via the kernels $H_{t}^{\alpha,\beta}(\theta,\varphi)$ and $\widetilde{H}{}_{t}^{\alpha,\beta}(\theta,\varphi)$, see \eqref{deco} and \eqref{jadro}.
Recall that $H{}_{t}^{\alpha,\beta}(\theta,\varphi)$ is, up to the factor $1/2$, the Jacobi-Poisson kernel considered in \cite{NoSjogren}.

The following symmetric and nonnegative expression
for $H{}_{t}^{\alpha,\beta}(\theta,\varphi)$ was derived in \cite[Proposition 4.1]{NoSjogren} and is based on the product formula 
for Jacobi polynomials due to Dijksma and Koornwinder \cite{DK}. For each $\alpha,\beta\ge-1/2$,

\begin{equation} \label{PJker}
H_t^{\ab}(\theta,\varphi) = c_{\alpha,\beta} \,\sinh\frac{t}{2}
	\iint \frac{d\Pi_{\alpha}(u)\,d\Pi_{\beta}(v)}
	{(\cosh\frac{t}{2}-1 + q(\theta,\varphi,u,v))^{\alpha+\beta+2}},
\end{equation}
where $c_{\alpha,\beta} =2^{-\alpha-\beta-2}\slash \m_{\ab}^{+}(0,\pi)$, the auxiliary function $q$ is given by
$$
q(\theta,\varphi,u,v) = 1 - u \sin\frac{\theta}2 \sin\frac{\varphi}2
	- v \cos\frac{\theta}2 \cos\frac{\varphi}2,
$$ 
and $\Pi_{\alpha}$ is a probability measure on the interval $[-1,1]$
defined for $\alpha>-1\slash 2$ by
$$
d\Pi_{\alpha}(u) = \frac{\Gamma(\alpha+1)}{\sqrt{\pi}\Gamma(\alpha+1\slash 2)}
	(1-u^2)^{\alpha-1\slash 2} du.
$$
In the limit case $\alpha=-1\slash 2$, we put
$$
\Pi_{-1\slash 2} = \frac{1}{2}(\eta_{-1}+\eta_{1}),
$$ where $\eta_{-1}$ and $\eta_{1}$ denote point masses at $-1$ and $1$, respectively.
Note that $\Pi_{-1\slash 2}$ is the weak limit of $\Pi_{\alpha}$ as $\alpha\rightarrow(-1/2)^{+}$. An immediate consequence of (\ref{PJker}) and \eqref{jadro} is the representation
\begin{equation*}
\widetilde{H}_t^{\ab}(\theta,\varphi) = \widetilde{c}_{\alpha,\beta} \,\sin\theta\,\sin\varphi\,\sinh\frac{t}{2}
	\iint \frac{d\Pi_{\alpha+1}(u)\,d\Pi_{\beta+1}(v)}
	{(\cosh\frac{t}{2}-1 + q(\theta,\varphi,u,v))^{\alpha+\beta+4}},
\end{equation*}
where $\widetilde{c}_{\alpha,\beta}=c_{\alpha+1,\beta+1}/4.$

To prove Theorem \ref{thm:stand} we will need some preparatory results. The first of them describes the measure of the interval $B(\theta,|\theta-\varphi|)$ and is valid for all $\alpha,\beta > -1.$  
\begin{lem}[{\cite[Lemma 4.2]{NoSjogren}}] \label{ball}
For all $\theta,\varphi \in (0,\pi)$, one has 
$$
\m_{\ab}^{+}\big( B(\theta,|\theta-\varphi|)\big) \simeq |\theta-\varphi| (\theta+\varphi)^{2\alpha+1}
	(\pi-\theta+\pi-\varphi)^{2\beta+1}.
$$
\end{lem}
\begin{cor}\label{cor:compare}
Let $\alpha,\beta>-1$. Given any $\gamma_{1},\gamma_{2}\ge0$, we have
$$
\mu_{\alpha+\gamma_{1},\beta+\gamma_{2}}^{+}(B(\theta,|\theta-\varphi|))\simeq(\theta+\varphi)^{2\gamma_{1}}(\pi-\theta+\pi-\varphi)^{2\gamma_{2}}\mu_{\alpha,\beta}^{+}(B(\theta,|\theta-\varphi|)),
$$
uniformly in $\theta,\varphi\in(0,\pi).$
\end{cor}

A convenient method of estimating kernels defined via the Jacobi-Poisson kernel was presented in \cite[Section 4]{NoSjogren}. The key point of this technique is the following result, which establishes a connection between bounds emerging from the representation (\ref{PJker}) and the standard estimates related to the space of homogeneous type $((0,\pi),d\m_{\ab}^{+},|\cdot|)$.
\begin{lem}[{\cite[Lemma 4.3]{NoSjogren}}] \label{bridge}
Let $\alpha,\beta \ge -1\slash 2$. Then
\begin{align*}
\iint \frac{d\Pi_{\alpha}(u)\,d\Pi_{\beta}(v)}{(q(\theta,\varphi,u,v))^{\alpha+\beta+3\slash 2}}
	& \lesssim \frac{1}{\m_{\ab}^{+}(B(\theta,|\theta-\varphi|))}, \qquad \theta,\varphi \in (0,\pi), 
	\quad \theta\neq \varphi,\\
	\iint \frac{d\Pi_{\alpha}(u)\,d\Pi_{\beta}(v)}{(q(\theta,\varphi,u,v))^{\alpha+\beta+2}} & \lesssim
	\frac{1}{|\theta-\varphi|\m_{\ab}^{+}(B(\theta,|\theta-\varphi|))}, 
	\qquad \theta,\varphi \in (0,\pi), \quad \theta\neq \varphi.
\end{align*}
\end{lem}

For our present considerations we generalize this result as follows.
\begin{lem} \label{L4.3*}
Let $\alpha,\beta\ge-1/2$ and $\gamma_{1},\gamma_{2},\kappa,\kappa_{1},\kappa_{2}\ge0$.
Then
\begin{align*}
&\Big(\sin\frac{\theta}{2}+\sin\frac{\varphi}{2}\Big)^{2\gamma_{1}}\Big(\cos\frac{{\theta}}{2}+\cos\frac{\varphi}{2}\Big)^{2\gamma_{2}}\iint\frac{d\Pi_{\alpha+\gamma_{1}+\kappa+\kappa_{1}}(u)\,d\Pi_{\beta+\gamma_{2}+\kappa+\kappa_{2}}(v)}{q(\theta,\varphi,u,v)^{\alpha+\beta+3/2+\gamma_{1}+\gamma_{2}+\kappa}}\\ 
&\lesssim\frac{1}{\mu_{\alpha,\beta}^{+}(B(\theta,|\theta-\varphi|))},\\
&\Big(\sin\frac{\theta}{2}+\sin\frac{\varphi}{2}\Big)^{2\gamma_{1}}\Big(\cos\frac{{\theta}}{2}+\cos\frac{\varphi}{2}\Big)^{2\gamma_{2}}\iint\frac{d\Pi_{\alpha+\gamma_{1}+\kappa+\kappa_{1}}(u)\,d\Pi_{\beta+\gamma_{2}+\kappa+\kappa_{2}}(v)}{q(\theta,\varphi,u,v)^{\alpha+\beta+2+\gamma_{1}+\gamma_{2}+\kappa}}\\ 
&\lesssim\frac{1}{|\theta-\varphi|\mu_{\alpha,\beta}^{+}(B(\theta,|\theta-\varphi|))},
\end{align*}
uniformly in $\theta,\varphi\in(0,\pi), \theta\neq\varphi.$
\end{lem}

\begin{proof}

We shall show the first estimate. Proving the second one is analogous. To begin with, observe that
$$\int \phi(u)\,d\Pi_{\alpha+\kappa_{1}}(u)\lesssim\int \phi(u)\,d\Pi_{\alpha}(u),$$ uniformly in all nonnegative and increasing functions $\phi$ on $[-1,1]$. Indeed, this is clear for $\alpha>-1/2$ since then the density of $\Pi_{\alpha+\kappa_{1}}$ is controlled by that of $\Pi_{\alpha}.$ On the other hand, if $\alpha=-1/2$ then we have $$\int \phi(u)\,d\Pi_{\alpha+\kappa_{1}}(u)\le\phi(1)\le2\int \phi(u)\,d\Pi_{\alpha}(u).$$ 

Applying this observation twice (the second time with $\alpha$ replaced by $\beta$ and $\kappa_{1}$ by $\kappa_{2}$), we see that
\begin{equation*}
\iint\frac{d\Pi_{\alpha+\kappa_{1}}(u)\,d\Pi_{\beta+\kappa_{2}}(v)}{q(\theta,\varphi,u,v)^{\alpha+\beta+3/2}}\lesssim\iint\frac{d\Pi_{\alpha}(u)\,d\Pi_{\beta}(v)}{q(\theta,\varphi,u,v)^{\alpha+\beta+3/2}}.
\end{equation*}
From here the conclusion in the special case when $\gamma_1=\gamma_2=\kappa=0$ follows by Lemma \ref{bridge}. 

Now let $\kappa,\gamma_1,\gamma_2\ge 0$ be arbitrary. 
Using twice the already verified special case (first with $\widetilde{\alpha}=\alpha+\gamma_{1}+\kappa,\widetilde{\beta}=\beta+\gamma_{2},\widetilde{\kappa}_1=\kappa_1,\widetilde{\kappa}_2=\kappa+\kappa_2$ and then with $\widetilde{\alpha}=\alpha+\gamma_{1},\widetilde{\beta}=\beta+\gamma_{2}+\kappa,\widetilde{\kappa}_1=\kappa+\kappa_1,\widetilde{\kappa}_2=\kappa_2$) we get the estimates

\begin{align*}
\iint\frac{d\Pi_{\alpha+\gamma_{1}+\kappa+\kappa_1}(u)\;d\Pi_{\beta+\gamma_{2}+\kappa+\kappa_2}(v)}{q(\theta,\varphi,u,v)^{\alpha+\beta+3/2+\gamma_{1}+\gamma_{2}+\kappa}}
\lesssim\frac{1}{(\theta+\varphi)^{2(\gamma_{1}+\kappa)}(\pi-\theta+\pi-\varphi)^{2\gamma_{2}}\mu_{\alpha,\beta}^{+}(B(\theta,|\theta-\varphi|))}
\end{align*}
and
\begin{align*}
\iint\frac{d\Pi_{\alpha+\gamma_{1}+\kappa+\kappa_1}(u)\;d\Pi_{\beta+\gamma_{2}+\kappa+\kappa_2}(v)}{q(\theta,\varphi,u,v)^{\alpha+\beta+3/2+\gamma_{1}+\gamma_{2}+\kappa}}\lesssim\frac{1}{(\theta+\varphi)^{2\gamma_{1}}(\pi-\theta+\pi-\varphi)^{2(\gamma_{2}+\kappa)}\mu_{\alpha,\beta}^{+}(B(\theta,|\theta-\varphi|))}.
\end{align*}
Taking minimum of the right-hand sides here, we see that
\begin{align*}
\iint\frac{d\Pi_{\alpha+\gamma_{1}+\kappa+\kappa_1}(u)\;d\Pi_{\beta+\gamma_{2}+\kappa+\kappa_2}(v)}{q(\theta,\varphi,u,v)^{\alpha+\beta+3/2+\gamma_{1}+\gamma_{2}+\kappa}}
&\lesssim\frac{1}{(\theta+\varphi)^{2\gamma_{1}}(\pi-\theta+\pi-\varphi)^{2\gamma_{2}}\mu_{\alpha,\beta}^{+}(B(\theta,|\theta-\varphi|))}.
\end{align*}
Since
$\theta\simeq\sin\frac{\theta}{2}$ and $\pi-\theta\simeq\cos\frac{\theta}{2}$ for $\theta\in(0,\pi)$, this readily implies our assertion.
\end{proof}

The next lemma will come into play when proving the smoothness estimates \eqref{sm1} and \eqref{sm2} for the relevant vector-valued kernels. It will allow us to reduce the difference conditions to certain gradient estimates, which are somewhat easier to verify.

\begin{lem}[{\cite[Lemma 4.6]{NoSjogren}}] \label{lem:comp}
For all $\theta,\widetilde{\theta},\varphi \in (0,\pi)$ with 
$|\theta-\varphi|>2|\theta-\widetilde{\theta}|$ and all $u,v \in [-1,1]$,
$$
q(\theta,\varphi,u,v) \simeq q(\widetilde{\theta},\varphi,u,v).
$$
Similarly, for all $\theta,\varphi,\widetilde{\varphi} \in (0,\pi)$ with 
$|\theta-\varphi|>2|\varphi-\widetilde{\varphi}|$ and all $u,v \in [-1,1]$,
$$
q(\theta,\varphi,u,v) \simeq q(\theta,\widetilde{\varphi},u,v).
$$
\end{lem}
We will also frequently make use of the bounds stated below.
\begin{lem}[{\cite[Lemma 4.5]{NoSjogren}}] \label{trig}
For all $\theta,\varphi \in (0,\pi)$ and $u,v \in [-1,1]$, one has
$$
\big| \partial_{\theta} q(\theta,\varphi,u,v) \big| \lesssim \sqrt{q(\theta,\varphi,u,v)} 
\qquad \textrm{and} \qquad
\big| \partial_{\varphi} q(\theta,\varphi,u,v) \big| \lesssim \sqrt{q(\theta,\varphi,u,v)}.
$$
\end{lem}

\begin{lem}\label{estimates}
Let $\theta, \varphi\in(0,\pi).$ Then
\begin{itemize}
\item[(a)]$$\frac{|\theta-\varphi|\varphi(\pi-\varphi)}{(\theta+\varphi)^{2}(\pi-\theta+\pi-\varphi)^{2}}\le1;$$
\item[(b)]$$\frac{\widetilde{\theta}\,\varphi \,(\pi-\widetilde{\theta})(\pi-\varphi)}{(\theta+\varphi)^{2}(\pi-\theta+\pi-\varphi)^{2}}\le1$$
provided that $2|\theta-\widetilde{\theta}|\le|\theta-\varphi|$.
\end{itemize}
These estimates remain valid after switching the roles of $\theta$ and $\varphi$.

\end{lem}
\begin{proof}
Simple exercise.
\end{proof}
Now we are in a position to prove Theorem \ref{thm:stand}. For the sake of brevity, we shall often use the shortened notation $$\q:= q(\theta,\varphi,u,v).$$ However, we shall never neglect the arguments if they are different from the above or their order is different. When proving Theorem \ref{thm:stand} we tacitly assume that changing orders of certain differentiations and integrations is legitimate. In fact, such manipulations can be easily justified with the aid of the estimates obtained along the proof of Theorem \ref{thm:stand} and the dominated convergence theorem.

To begin with, we consider the kernels of the maximal operators $({H}_{*}^{\alpha,\beta})^{+}$ and $(\widetilde{H}_{*}^{\alpha,\beta})^{+}$ and observe that the case of $({H}_{*}^{\alpha,\beta})^{+}$ was already treated in \cite{NoSjogren}. It remains to analyze the kernel 
$\{\widetilde{H}_t^{\ab}(\theta,\varphi)\}_{t>0}$.

\begin{proof}[Proof of Theorem \ref{thm:stand}; the case of $\{\widetilde{H}_t^{\ab}(\theta,\varphi)\}_{t>0}$]
We first deal with the growth condition \eqref{gr} specified to $\mathbb{B}=\mathbb{X}\subset L^{\infty}$. 
As mentioned above, in \cite[Section 4]{NoSjogren} the vector-valued kernel $\{H{}_{t}^{\alpha,\beta}(\theta,\varphi)\}_{t>0}$ was shown to
satisfy \eqref{gr}. Thus 
\begin{align*}
\big\|{\widetilde{H}}{}_{t}^{\alpha,\beta}(\theta,\varphi)\big\|_{L^{\infty}(dt)}=\frac{\sin\theta\,\sin\varphi}{4}\,\big\|H{}_{t}^{\alpha+1,\beta+1}(\theta,\varphi)\big\|_{L^{\infty}(dt)}\lesssim\frac{\sin\theta\,\sin\varphi}{\mu_{\alpha+1,\beta+1}^{+}(B(\theta,|\theta-\varphi|))}.
\end{align*} 
Since 
\begin{align*}
\sin\theta\,\sin\varphi\simeq\theta\varphi(\pi-\theta)(\pi-\varphi)\le(\theta+\varphi)^{2}(\pi-\theta+\pi-\varphi)^{2},
\end{align*}
from Corollary \ref{cor:compare} (taken with $\gamma_{1}=\gamma_{2}=1$) it follows that
\begin{align*}
\big\|\widetilde{H}{}_{t}^{\alpha,\beta}(\theta,\varphi)\big\|_{L^{\infty}(dt)}\lesssim\frac{1}{\mu_{\alpha,\beta}^{+}(B(\theta,|\theta-\varphi|))}.
\end{align*} 

We pass to proving the smoothness conditions. For symmetry reasons it is enough to show (\ref{sm1}). Let $\theta,\theta',\varphi$ be such that $2|\theta-\theta'|<|\theta-\varphi|$.
By the Mean Value Theorem, we have
\begin{align*} 
\big|\widetilde{H}{}_{t}^{\alpha,\beta}(\theta,\varphi)-\widetilde{H}{}_{t}^{\alpha,\beta}(\theta',\varphi)\big|=|\theta-\theta'|\big|\partial_{\theta}\widetilde{H}{}_{t}^{\alpha,\beta}(\widetilde{\theta},\varphi)\big|,
\end{align*}
where $\widetilde{\theta}$ is a convex combination of $\theta$ and $\theta'$ (notice that $\widetilde{\theta}$ may depend also on $t$ and $\varphi$). Therefore
\begin{align*} \big|\widetilde{H}{}_{t}^{\alpha,\beta}(\theta,\varphi)-\widetilde{H}{}_{t}^{\alpha,\beta}(\theta',\varphi)\big|&\le|\theta-\theta'|\,|\cos\widetilde{\theta}|\,\sin\varphi\,H{}_{t}^{\alpha+1,\beta+1}(\widetilde{\theta},\varphi)\\ &\quad+|\theta-\theta'|\sin\widetilde{\theta}\,\sin\varphi\,\big|\partial_{\theta}H{}_{t}^{\alpha+1,\beta+1}(\widetilde{\theta},\varphi)\big|\\
&\equiv |\theta-\theta'|(S_{1}+S_{2}).
\end{align*}
Our task will be done once we show that 
\begin{equation*}
S_{i}\lesssim\frac{1}{|\theta-\varphi|\mu_{\alpha,\beta}^{+}(B(\theta,|\theta-\varphi|))},\qquad  i=1,2,\\
\end{equation*}
uniformly in $t>0$ and $\theta, \varphi\in(0,\pi)$ such that $2|\theta-\theta'|<|\theta-\varphi|.$

Using \eqref{PJker} and Lemma \ref{lem:comp}, we see that
\begin{align*}
S_{1}&\lesssim\sin\varphi\,\sinh\frac{t}2 \iint \frac{d\Pi_{\alpha+1}(u)\,d\Pi_{\beta+1}(v)} {(\cosh\frac{t}2-1 + q(\widetilde{\theta},\varphi,u,v))^{\alpha+\beta+4}}\\ 
&\simeq\sin\varphi\,H{}_{t}^{\alpha+1,\beta+1}(\theta,\varphi).
\end{align*}
Now, applying (\ref{gr}) for the kernel $\{H{}_{t}^{\alpha+1,\beta+1}(\theta,\varphi)\}_{t>0}$ and then using Corollary \ref{cor:compare} and the relation $\sin\varphi\simeq\varphi(\pi-\varphi),\;\varphi\in(0,\pi)$, we get
\begin{align*}
S_{1}&\lesssim\frac{\sin\varphi}{\mu_{\alpha+1,\beta+1}^{+}(B(\theta,|\theta-\varphi|))}\simeq\frac{\varphi(\pi-\varphi)}{(\theta+\varphi)^{2}(\pi-\theta+\pi-\varphi)^{2}}\frac{1}{\mu_{\alpha,\beta}^{+}(B(\theta,|\theta-\varphi|))}.
\end{align*}
This, in view of Lemma \ref{estimates} (a), implies
\begin{equation*}
S_{1}\lesssim\frac{1}{|\theta-\varphi|\mu_{\alpha,\beta}^{+}(B(\theta,|\theta-\varphi|))}.
\end{equation*}

Considering $S_{2}$, taking into account \eqref{PJker}, Lemma \ref{trig} and Lemma \ref{lem:comp}, we have
\begin{align*}
S_{2}&\lesssim\sin\widetilde{\theta}\,\sin\varphi\,\sinh\frac{t}{2} \iint\frac{|\partial_{\theta}q(\widetilde{\theta},\varphi,u,v)|}{(\cosh\frac{t}2-1+q(\widetilde{\theta},\varphi,u,v))^{\alpha+\beta+5}}\,{d\Pi_{\alpha+1}(u)\,d\Pi_{\beta+1}(v)} \\
&\lesssim\sin\widetilde{\theta}\,\sin\varphi\,\sinh\frac{t}{2} \iint  \frac{d\Pi_{\alpha+1}(u)\,d\Pi_{\beta+1}(v)}{(\cosh\frac{t}{2}-1+q(\widetilde{\theta},\varphi,u,v))^{\alpha+\beta+9/2}}\\
&\lesssim\sin\widetilde{\theta}\,\sin\varphi\,\sinh\frac{t}{2}\iint  \frac{d\Pi_{\alpha+1}(u)\,d\Pi_{\beta+1}(v)}{(\cosh\frac{t}{2}-1+\q)^{\alpha+\beta+9/2}}
\end{align*}
for $2|\theta-\theta'|<|\theta-\varphi|.$ Observe that 
\begin{equation}\label{asympt}
\frac{\sinh\frac{t}{2}}{(\cosh\frac{t}{2}-1+\q)^{\alpha+\beta+9/2}}\lesssim\frac{1}{\q^{\alpha+\beta+4}},
\end{equation}
 uniformly in $\q$ and $t>0$. For $t$ small this follows by the asymptotics 
$\cosh\frac{t}2-1\simeq t^2$ and $\sinh\frac{t}2\simeq t$, and for large $t$ we use the asymptotics $\tanh\frac{t}2\simeq 1$ and boundedness of the quantity $\q$. Combining \eqref{asympt} with the last estimate of $S_{2}$ and Lemma \ref{bridge} leads to the bound
$$S_{2}\lesssim\frac{\sin\widetilde{\theta}\,\sin\varphi}{|\theta-\varphi|\mu_{\alpha+1,\,\beta+1}^{+}(B(\theta,|\theta-\varphi|))}.$$
Now, using Corollary \ref{cor:compare} we infer that
\begin{align*}
S_{2}\lesssim\frac{\widetilde{\theta}\,\varphi(\pi-\widetilde{\theta})(\pi-\varphi)}{(\theta+\varphi)^{2}(\pi-\theta+\pi-\varphi)^{2}}\,\frac{1}{|\theta-\varphi|\mu_{\alpha,\beta}^{+}(B(\theta,|\theta-\varphi|))}.
\end{align*}
Finally, since $2|\theta-\widetilde{\theta}|<|\theta-\varphi|,$ Lemma \ref{estimates} (b) gives the desired estimate of $S_2$. This finishes proving the case of $\{\widetilde{H}_t^{\ab}(\theta,\varphi)\}_{t>0}$ in Theorem \ref{thm:stand}.
\end{proof}

The next kernels to be considered are those of Riesz-Jacobi type transforms of arbitrary order $N\ge 1$.
To proceed we will need a technical result from \cite{NoSjogren}. Actually, the lemma below could have been used earlier in the proof of Theorem \ref{thm:stand}. However, for the sake of better understanding, we decided to postpone it a bit to let the reader intrinsically trace all the details. Denote
$$
\Psi^{\ab}(t,\q) = \frac{\sinh\frac{t}2}{(\cosh\frac{t}2-1+\q)^{\alpha+\beta+2}}.
$$
Notice that in view of (\ref{PJker}), this expression integrated against
$d\Pi_{\alpha}(u)\, d\Pi_{\beta}(v)$ gives, up to a constant factor, the kernel $H_t^{\ab}(\theta,\varphi)$.

\begin{lem}[{\cite[Lemma 4.8]{NoSjogren}}]  \label{4.8}
Let $\alpha,\beta \ge -1\slash 2$ and $m,n\ge 0$ be given. Then

\begin{align*}
\big|\partial_{\theta}^n \partial_t^{m} \Psi^{\ab}(t,\q)\big|\lesssim
\begin{cases} 
(t^2+\q)^{-\alpha-\beta-3\slash 2-(m+n)/2}, & t \le 1 \\
\exp(-\zeta t), & t > 1
\end{cases},
\end{align*}
where $\zeta=\zeta(\alpha,\beta)>0$ unless $m=n=0$ and $\alpha+\beta=-1;$ in the latter case $\zeta=0.$ Moreover,
\begin{equation*} 
\big|\partial_{\varphi}\partial_{\theta}^n \partial_t^{m} \Psi^{\ab}(t,\q)\big| \lesssim
\begin{cases}
(t^2+\q)^{-\alpha-\beta- 2-(m+n)/2}, & t \le 1 \\
\exp(-t/4), & t > 1
\end{cases}.
\end{equation*}
\end{lem}

First we show the standard estimates for $R_{N}^{\alpha,\beta}(\theta,\varphi)$.
\begin{proof}[Proof of Theorem \ref{thm:stand}; the case of $R_{N}^{\ab}(\theta,\varphi)$]

Assume first that $N$ is an even number and let $k_{0}=N/2\ge 1$. By term by term differentiation of the defining series it can be verified that $H_{t}^{\alpha,\beta}(\theta,\varphi)$ satisfies in the strip $(t,\theta)\in(0,\infty)\times(0,\pi)$ the Laplace equation based on the Jacobi `Laplacian', which can be written as
\begin{equation*}
\delta_2^{\textrm{even}}H_{t}^{\alpha,\beta}(\theta,\varphi)=\partial_{t}^{2}H_{t}^{\alpha,\beta}(\theta,\varphi)-\lambda_0^{\ab}H_{t}^{\alpha,\beta}(\theta,\varphi).
\end{equation*}
 Iterating this relation, we get 
 \begin{align}\label{pochodne}
 \de H_{t}^{\alpha,\beta}(\theta,\varphi)=\sum_{j=0}^{k_{0}}c_{j}\,\partial_{t}^{2(k_{0}-j)}H_{t}^{\alpha,\beta}(\theta,\varphi) 
 \end{align}
 with some constants $c_{j}=c_{j}(\alpha,\beta, k_{0}).$
 Consequently, it suffices to show that for each $k=0,1,...\,,k_{0}$ the kernel 
 \begin{equation}\label{RieszNorm}
 S_{k}^{\alpha,\beta}(\theta,\varphi):=\int_{0}^{\infty}\partial_{t}^{2k}H_{t}^{\alpha,\beta}(\theta,\varphi)\,t^{2k_{0}-1}\,dt
 \end{equation}
  satisfies conditions (\ref{gr}) and (\ref{sm}). Moreover, in the critical case $\alpha=\beta=-1/2$ the constant $\lambda_0^{\ab}$ equals $0$, formula \eqref{pochodne} becomes   
 \begin{align*}
 \de H_{t}^{-1/2,-1/2}(\theta,\varphi)=\partial_{t}^{2k_{0}}H_{t}^{\alpha,\beta}(\theta,\varphi) 
 \end{align*}
and therefore in this case it is enough to estimate $S_{k_{0}}^{-1/2,-1/2}(\theta,\varphi).$
 
We have 
\begin{align*}
|S_{k}^{\alpha,\beta}(\theta,\varphi)|&\lesssim\iint\d\Pi_{\alpha}(u)\,d\Pi_{\beta}(v)\,\int_{0}^{\infty}|\partial_{t}^{2k}\Psi^{\ab}(t,\q)|\,t^{2k_{0}-1}\,dt\\
&\equiv\iint\d\Pi_{\alpha}(u)\,d\Pi_{\beta}(v)\,(I_{0}+I_{\infty}),
\end{align*}
where $I_{0}$ and $I_{\infty}$ are the integrals in $t$ 
over $(0,1)$ and $(1,\infty)$, respectively. To estimate the first of these integrals we use Lemma \ref{4.8} and the boundedness of $\q$, and then change the variable $t=\sqrt{\q}s$, getting
\begin{align*}
I_{0}&\lesssim\int_{0}^{1}\frac{t^{2k_{0}-1}\,dt}{(t^{2}+\q)^{\alpha+\beta+3/2+k_0}}\le\frac{1}{\q^{\alpha+\beta+3/2}}\int_{0}^{\infty}\frac{s^{2k_{0}-1}\,ds}{(1+s^{2})^{\alpha+\beta+3/2+k_0}}\lesssim\frac{1}{\q^{\alpha+\beta+3/2}}.
\end{align*}
Furthermore, using again Lemma \ref{4.8} and the boundedness of $\q$, we also see that
$$
I_{\infty} \lesssim \int_1^{\infty} \frac{t^{2k_{0}-1}dt}{e^{\zeta t}}\simeq 1 \lesssim 
	\frac{1}{\q^{\alpha+\beta+3/2}},
$$
with some constant $\zeta=\zeta(\alpha,\beta)>0$ (recall that when $\alpha+\beta=-1$ we consider only $k=k_0\ge1$). Together with Lemma \ref{bridge} these estimates give the growth condition (\ref{gr}) for each $S_{k}^{\alpha,\beta}(\theta,\varphi)$, $k=0,1,...,k_{0}$ and thus also for $R_{2k_{0}}^{\ab}(\theta,\varphi)$.

We pass to proving the gradient bound (\ref{sm}) for $S_{k}^{\alpha,\beta}(\theta,\varphi)$. For symmetry reasons it is enough to show that 
$$|\partial_{\theta}S_{k}^{\alpha,\beta}(\theta,\varphi)|\lesssim\frac{1}{|\theta-\varphi|}\,\frac{1}{\mu_{\alpha,\beta}^{+}(B(\theta,|\theta-\varphi|))}$$ 
for each $0\le k\le k_{0}.$
We have
\begin{align*}
|\partial_{\theta}S_{k}^{\alpha,\beta}(\theta,\varphi)|&\lesssim\iint\d\Pi_{\alpha}(u)\,d\Pi_{\beta}(v)\,\int_{0}^{\infty}|\partial_{\theta}\,\partial_{t}^{2k}\,\Psi^{\ab}(t,\q)|\,t^{2k_{0}-1}\,dt\\ 
&\equiv\iint\d\Pi_{\alpha}(u)\,d\Pi_{\beta}(v)\,(J_{0}+J_{\infty}).
\end{align*}
Similarly as in case of the growth condition, we use Lemma \ref{4.8} to bound $J_0$ and $J_{\infty}$ by $\q^{-\alpha-\beta-2}$. Then an application of Lemma \ref{bridge} leads to the desired conclusion.

Now let $N$ be odd and take $k_{0}\ge0$ such that $N=2k_{0}+1$. Then
\begin{align*}
\de H_{t}^{\alpha,\beta}(\theta,\varphi)=\sum_{j=0}^{k_{0}}c_{j}\,\partial_{\theta}\,\partial_{t}^{2(k_{0}-j)}H_{t}^{\alpha,\beta}(\theta,\varphi).
\end{align*}
Consequently, we see that proving (\ref{gr}) and (\ref{sm}) for the kernel in question may be reduced to doing the same for each
\begin{equation*}
\mathcal{S}_{k}^{\alpha,\beta}(\theta,\varphi):=\int_{0}^{\infty}\partial_{\theta}\partial_{t}^{2k}H_{t}^{\alpha,\beta}(\theta,\varphi)\,t^{2k_{0}}\,dt,\qquad  0\le k\le k_{0}.
\end{equation*}

We have
\begin{align*}
|\mathcal{S}_{k}^{\alpha,\beta}(\theta,\varphi)|&\lesssim\iint\d\Pi_{\alpha}(u)\,d\Pi_{\beta}(v)\,\int_{0}^{\infty}|\partial_{\theta}\partial_{t}^{2k}\Psi^{\ab}(t,\q)|t^{2k_{0}}dt\\ 
&\equiv\iint\d\Pi_{\alpha}(u)\,d\Pi_{\beta}(v)\,(\mathcal{I}_{0}+\mathcal{I}_{\infty}),
\end{align*}
where $\mathcal{I}_{0}$ and $\mathcal{I}_{\infty}$ are the integrals in $t$ over $(0,1)$ and $(1,\infty)$, respectively.
Using Lemma \ref{4.8}, the boundedness of $\q$, and changing the variable $t=\sqrt{\q}s,$ we obtain
\begin{align*}
\mathcal{I}_{0}&\lesssim\int_{0}^{1}\frac{t^{2k_{0}}\,dt}{(t^{2}+\q)^{\alpha+\beta+2+k_0}}\le\frac{1}{\q^{\alpha+\beta+3/2}}\int_{0}^{\infty}\frac{s^{2k_0}\,ds}{(1+s^{2})^{\alpha+\beta+2+k_0}}\lesssim\frac{1}{\q^{\alpha+\beta+3/2}}.
\end{align*}
On the other hand, by Lemma \ref{4.8} and the boundedness of $\q$
$$
\mathcal{I}_{\infty} \lesssim \int_1^{\infty} \frac{t^{2k_{0}}\,dt}{e^{t/4}}\lesssim\frac{1}{\q^{\alpha+\beta+3/2}}.
$$
The growth bound for $\mathcal{S}_{k}^{\alpha,\beta}(\theta,\varphi)$, $0\le k\le k_{0}$, follows now from Lemma \ref{bridge}.

Considering the gradient condition, we want to show that for each $0\le k\le k_{0}$
\begin{align*}
|\partial_{\theta}\mathcal{S}_{k}^{\alpha,\beta}(\theta,\varphi)|+|\partial_{\varphi}\mathcal{S}_{k}^{\alpha,\beta}(\theta,\varphi)|\lesssim\frac{1}{|\theta-\varphi|}\,\frac{1}{\mu_{\alpha,\beta}^{+}(B(\theta,|\theta-\varphi|))}.
\end{align*}
Having in mind Lemma \ref{bridge}, we see that it is enough to verify that both terms on the left-hand side above are controlled by $$\iint\frac{\d\Pi_{\alpha}(u)\,d\Pi_{\beta}(v)}{\q^{\alpha+\beta+2}}.$$ 
This, however, follows with the aid of Lemma \ref{4.8} and the arguments already presented in this proof. Thus we omit the details.
\end{proof}
 
 In a similar spirit we deal with $\widetilde{R}_{N}^{\alpha,\beta}(\theta,\varphi),$ but the estimates are slightly more complicated.

\begin{proof}[Proof of Theorem \ref{thm:stand}; the case of $\widetilde{R}_{N}^{\ab}(\theta,\varphi)$] 
 Assume first that $N\ge 2$ is even and write it as $N=2k_{0}$ with $k_{0}\ge 1$. Term by term differentiation of the series defining $\widetilde{H}_{t}^{\alpha,\beta}(\theta,\varphi)$ allows one to verify that this kernel satisfies in the strip $(t,\theta)\in(0,\infty)\times(0,\pi)$ the Laplace equation based on the modified Jacobi operator $\delta\delta^*+\lambda_0^{\ab}.$ This can be written as
 \begin{equation}\label{delta2}
\delta_{2}^{\textrm{odd}}\widetilde{H}_{t}^{\alpha,\beta}(\theta,\varphi)=\partial_{t}^{2}\widetilde{H}_{t}^{\alpha,\beta}(\theta,\varphi)-\lambda_0^{\ab}\widetilde{H}_{t}^{\alpha,\beta}(\theta,\varphi).
\end{equation}
  Iterating this identity we get, for some constants $c_j$,
 \begin{align*}
 \widetilde{R}_{N}^{\alpha,\beta}(\theta,\varphi)=\int_{0}^{\infty}\sum_{j=0}^{k_{0}}c_{j}\,\partial_{t}^{2(k_{0}-j)}\widetilde{H}_{t}^{\alpha,\beta}(\theta,\varphi)\,t^{2k_{0}-1}\,dt.
 \end{align*}
 It is clear that our task will be done once we show that for each $0\le k\le k_{0}$ the kernel
 $$\widetilde{S}_{k}^{\alpha,\beta}(\theta,\varphi):=\int_{0}^{\infty}\partial_{t}^{2k}\widetilde{H}_{t}^{\alpha,\beta}(\theta,\varphi)\,t^{2k_{0}-1}\,dt$$
 satisfies the standard estimates.
 
Proving the growth condition (\ref{gr}) is nearly straightforward, because
\begin{align*}
|\widetilde{S}_{k}^{\alpha,\beta}(\theta,\varphi)|=\frac{\sin\theta\,\sin\varphi}{4}\,\bigg|\int_{0}^{\infty}\partial_{t}^{2k}H_{t}^{\alpha+1,\beta+1}(\theta,\varphi)\,t^{2k_{0}-1}\,dt\bigg|.
\end{align*}
This combined with the already proved growth estimate for $S_{k}^{\alpha+1,\beta+1}(\theta,\varphi)$, see (\ref{RieszNorm}), Corollary \ref{cor:compare} and Lemma \ref{estimates} (b) with $\widetilde{\theta}=\theta$ produces
\begin{align*}
|\widetilde{S}_{k}^{\alpha,\beta}(\theta,\varphi)|&\lesssim\frac{\sin\theta\,\sin\varphi}{(\theta+\varphi)^{2}(\pi-\theta+\pi-\varphi)^{2}}\,\frac{(\theta+\varphi)^{2}(\pi-\theta+\pi-\varphi)^{2}}{\mu_{\alpha+1,\beta+1}^{+}(B(\theta,|\theta-\varphi|))}\lesssim\frac{1}{\mu_{\alpha,\beta}^{+}(B(\theta,|\theta-\varphi|))}.
\end{align*}  

We pass to the gradient condition (\ref{sm}).
For symmetry reasons, it is enough to consider only the derivative in $\theta$. A simple computation shows that 
\begin{align*}
|\partial_{\theta}\widetilde{S}_{k}^{\alpha,\beta}(\theta,\varphi)|&\le\sin\varphi\,\bigg|\int_{0}^{\infty}\partial_{t}^{2k}H_{t}^{\alpha+1,\beta+1}(\theta,\varphi)\,t^{2k_{0}-1}\,dt\bigg|\\
&\quad+\sin\theta\,\sin\varphi\,\bigg|\int_{0}^{\infty}\partial_{\theta}\partial_{t}^{2k} H_{t}^{\alpha+1,\beta+1}(\theta,\varphi)\,t^{2k_{0}-1}\,dt\bigg|\\
&\equiv V_{1}+V_{2}.
\end{align*}
Using again the growth estimate for $S_{k}^{\alpha+1,\beta+1}(\theta,\varphi)$ and then Lemma \ref{estimates} (a), we get
$$V_1\lesssim\frac{\sin\varphi}{\mu_{\alpha+1,\beta+1}^{+}(B(\theta,|\theta-\varphi|))}\lesssim\frac{1}{|\theta-\varphi|}\,\frac{1}{\mu_{\alpha,\beta}^{+}(B(\theta,|\theta-\varphi|))}.$$
Considering $V_{2}$, we observe that the gradient bound for $S_{k}^{\alpha+1,\beta+1}(\theta,\varphi)$ implies
 $$V_{2}\lesssim\frac{1}{|\theta-\varphi|}\,\frac{\sin\theta\sin\varphi}{\mu_{\alpha+1,\beta+1}^{+}(B(\theta,|\theta-\varphi|))}.$$
 This combined with Corollary \ref{cor:compare} and Lemma \ref{estimates} (b) with $\widetilde{\theta}=\theta$ gives 
$$V_{2}\lesssim\frac{1}{|\theta-\varphi|}\,\frac{1}{\mu_{\alpha,\beta}^{+}(B(\theta,|\theta-\varphi|))}.$$
The reasoning for the case of $N$ even is finished.

Now let $N$ be odd, $N=2k_{0}+1$ for some $k_{0}\ge 0$.
We have 
\begin{align*}
\widetilde{R}_{N}^{\alpha,\beta}(\theta,\varphi)&=\int_{0}^{\infty}\sum_{j=0}^{k_{0}}c_{j}\,\delta^{*}\partial_{t}^{2(k_{0}-j)}\widetilde{H}_{t}^{\alpha,\beta}(\theta,\varphi)\,t^{2k_{0}}\,dt.
\end{align*} 
Again, we observe that it is enough to show that, for each $0\le k\le k_{0}$, the kernel
$$\widetilde{\mathcal{S}}_{k}^{\alpha,\beta}(\theta,\varphi):=\int_{0}^{\infty}\delta^{*}\partial_{t}^{2k}\widetilde{H}^{\alpha,\beta}(\theta,\varphi)\,t^{2k_{0}}\,dt$$
satisfies the standard estimates (here and elsewhere $\delta^*$ acts always on $\theta$ variable).

A direct computation reveals that
\begin{align*}
|\widetilde{\mathcal{S}}_{k}^{\alpha,\beta}(\theta,\varphi)|&\lesssim\sin\theta\,\sin\varphi\,\int_{0}^{\infty}|\partial_{\theta}\partial_{t}^{2k}\,H_{t}^{\alpha+1,\beta+1}(\theta,\varphi)|\,t^{2k_{0}}\,dt\\ 
&\quad+\sin\varphi\,\int_{0}^{\infty}|\partial_{t}^{2k}\,H_{t}^{\alpha+1,\beta+1}(\theta,\varphi)|\,t^{2k_{0}}\,dt\\ 
&\equiv W_{1}+W_{2}.
\end{align*}
Estimates analogous to those in the proof of the growth condition for $\mathcal{S}_{k}^{\alpha+1,\beta+1}(\theta,\varphi)$ show that $$\int_{0}^{\infty}|\partial_{\theta}\partial_{t}^{2k}\,H_{t}^{\alpha+1,\beta+1}(\theta,\varphi)|\,t^{2k_{0}}dt\lesssim \iint\frac{d\Pi_{\alpha+1}(u)\,d\Pi_{\beta+1}(v)}{\q^{\alpha+\beta+7/2}}.$$ 
Then Lemma \ref{L4.3*} (with $\gamma_1=\gamma_2=1,\kappa=\kappa_1=\kappa_2=0$) and Lemma \ref{estimates} (b) (with $\widetilde{\theta}=\theta$) lead to
$$W_{1}\lesssim\frac{1}{\mu_{\alpha,\beta}^{+}(B(\theta,|\theta-\varphi|))}.$$
Proceeding in the same way as in the proof of the growth condition for $S_{k}^{\alpha,\beta}(\theta,\varphi)$ we get $$\int_{0}^{\infty}|\partial_{t}^{2k}\,H_{t}^{\alpha+1,\beta+1}(\theta,\varphi)|\,t^{2k_{0}}\,dt\lesssim\iint\frac{d\Pi_{\alpha+1}(u)\,d\Pi_{\beta+1}(v)}{\q^{\alpha+\beta+3}},$$
and applying Lemma \ref{L4.3*} (with $\gamma_{1}=\gamma_{2}=\kappa=1/2, \kappa_{1}=\kappa_{2}=0$) we obtain 
$$W_{2}\lesssim\frac{1}{\mu_{\alpha,\beta}^{+}(B(\theta,|\theta-\varphi|))}.$$
The growth estimate for $\widetilde{\mathcal{S}}_{k}^{\alpha,\beta}(\theta,\varphi)$, $0\le k\le k_0$, follows.

To commence the proof of the gradient condition we use \eqref{delta2} to write
 
\begin{align*}
|\partial_{\theta}\widetilde{\mathcal{S}}_{k}^{\alpha,\beta}(\theta,\varphi)|&\lesssim\sin\theta\,\sin\varphi\,\int_{0}^{\infty}|\partial_{t}^{2k+2}\,H_{t}^{\alpha+1,\beta+1}(\theta,\varphi)|\,t^{2k_{0}}\,dt\\ 
&\quad+\sin\theta\sin\varphi\,\int_{0}^{\infty}|\partial_{t}^{2k}\,H_{t}^{\alpha+1,\beta+1}(\theta,\varphi)|\,t^{2k_{0}}\,dt\\
&\equiv Z_{1}+Z_{2}.
\end{align*}
Proceeding in the same way as before (precisely, splitting the relevant inner integral in $t$ in a suitable way, using Lemma \ref{4.8} and then integrating by change of the variable) we obtain
$$\int_{0}^{\infty}|\partial_{t}^{2k+2}\,H_{t}^{\alpha+1,\beta+1}(\theta,\varphi)|\,t^{2k_{0}}\,dt\lesssim\iint\frac{d\Pi_{\alpha+1}(u)\,d\Pi_{\beta+1}(v)}{\q^{\alpha+\beta+4}}.$$
Similarly, by Lemma \ref{4.8} we conclude the bound
$$\int_{0}^{\infty}|\partial_{t}^{2k}\,H_{t}^{\alpha+1,\beta+1}(\theta,\varphi)|\,t^{2k_{0}}\,dt\lesssim\iint\frac{d\Pi_{\alpha+1}(u)\,d\Pi_{\beta+1}(v)}{\q^{\alpha+\beta+3}}\lesssim\iint\frac{d\Pi_{\alpha+1}(u)\,d\Pi_{\beta+1}(v)}{\q^{\alpha+\beta+4}}.$$
Applying now the second estimate from Lemma \ref{L4.3*} (taken with $\gamma_1=\gamma_2=1,\kappa=\kappa_1=\kappa_2=0$) and Lemma \ref{estimates} (b) (with $\widetilde{\theta}=\theta$) we get
\begin{align*}
Z_{1}+Z_{2}\lesssim\frac{1}{|\theta-\varphi|}\,\frac{1}{\mu_{\alpha,\beta}^{+}(B(\theta,|\theta-\varphi|))}.
\end{align*}

It remains to estimate the derivative in $\varphi$. We have
\begin{align*}
|\partial_{\varphi}\widetilde{\mathcal{S}}_{k}^{\alpha,\beta}(\theta,\varphi)|&\lesssim\sin\theta\,\sin\varphi\,\int_{0}^{\infty}|\partial_{\varphi}\partial_{\theta}\partial_{t}^{2k}\,H_{t}^{\alpha+1,\beta+1}(\theta,\varphi)|\,t^{2k_{0}}\,dt\\ 
&\quad+\sin\varphi\,\int_{0}^{\infty}|\partial_{\varphi}\partial_{t}^{2k}\,H_{t}^{\alpha+1,\beta+1}(\theta,\varphi)|\,t^{2k_{0}}\,dt\\ 
&\quad+\sin\theta\,\int_{0}^{\infty}|\partial_{\theta}\partial_{t}^{2k}\,H_{t}^{\alpha+1,\beta+1}(\theta,\varphi)|\,t^{2k_{0}}\,dt\\
&\quad+\int_{0}^{\infty}|\partial_{t}^{2k}\,H_{t}^{\alpha+1,\beta+1}(\theta,\varphi)|\,t^{2k_{0}}\,dt\\ 
&\equiv P_{1}+P_{2}+P_{3}+P_{4}. 
\end{align*}
In view of Lemma \ref{4.8} and earlier considerations, the following estimates hold:
\begin{align*}
P_1&\lesssim\sin\theta\,\sin\varphi\,\iint\frac{d\Pi_{\alpha+1}(u)\,d\Pi_{\beta+1}(v)}{\q^{\alpha+\beta+4}},\\
P_2&\lesssim\sin\varphi\,\iint\frac{d\Pi_{\alpha+1}(u)\,d\Pi_{\beta+1}(v)}{\q^{\alpha+\beta+7/2}},\\
P_3&\lesssim\sin\theta\,\iint\frac{d\Pi_{\alpha+1}(u)\,d\Pi_{\beta+1}(v)}{\q^{\alpha+\beta+7/2}},\\
P_4&\lesssim\iint\frac{d\Pi_{\alpha+1}(u)\,d\Pi_{\beta+1}(v)}{\q^{\alpha+\beta+3}}.
\end{align*}
Using these bounds, $P_{1}$ may be dealt with in the same way as $Z_1$ above. To treat $P_{2}$, $P_{3}$ and $P_4$ it suffices to apply the second estimate from Lemma \ref{L4.3*} with suitably chosen parameters.
Summing up, we conclude that for each $0\le k\le k_0$
\begin{align*}
|\partial_{\varphi}\widetilde{\mathcal{S}}_{k}^{\alpha,\beta}(\theta,\varphi)|&\lesssim\frac{1}{|\theta-\varphi|}\,\frac{1}{\mu_{\alpha,\beta}^{+}(B(\theta,|\theta-\varphi|))}.
\end{align*}

The proof of the case of $\widetilde{R}_N^{\ab}(\theta,\varphi)$ in Theorem \ref{thm:stand} is complete.
\end{proof}
The next kernels to estimate are those of the auxiliary $g$-functions. To this end, we denote $\mathbb{B}=L^2(t^{2M+2N-1}dt)$. Moreover, for the sake of clarity, we consider the essential range of $\alpha,\beta\ge-1/2$ satisfying $\alpha+\beta>-1$. The remaining critical case $\alpha=\beta=-1/2$ requires simplified arguments comparing to what follows (in particular, $\de H_t^{-1/2,-1/2}(\theta,\varphi)$ and $\ko \widetilde{H}_t^{-1/2,-1/2}(\theta,\varphi)$ can be expressed much easier via the $\partial_t$ derivatives), and it is left to the reader.

Given $\theta,\theta'\in(0,\pi),$ denote by $I(\theta,\theta')$ the open interval with endpoints $\theta$ and $\theta'$.
We shall need the following technical result, which is actually a consequence of Lemma \ref{4.8}.

\begin{lem}\label{4.8*}
Assume that $\alpha,\beta\ge -1/2$ and $\alpha+\beta> -1$. Let $M,N\ge 0, M+N>0$, $K\ge0$ and $\ep_1\in\{0,1,2\},\ep_2\in\{0,1\},\tau\in\{0,1\},\gamma\in\{\tau,2\tau\}$. The following estimates hold uniformly in $\theta, \varphi\in (0,\pi)$ or in $\theta, \theta', \varphi\in (0,\pi)$ satisfying $2|\theta-\theta'|<|\theta-\varphi|$, respectively.
\begin{itemize}
\item[(a)]
If $\ep_1+K+2\tau-\gamma\le M+N$, then
\begin{align*}
\big\|\partial_{\theta}^{\ep_1}\partial_{t}^{K}H_{t}^{\alpha+\tau,\beta+\tau}({\theta}, \varphi)\big\|_{\mathbb{B}}&\lesssim \frac{(\theta+\varphi)^{-\gamma}(\pi-\theta+\pi-\varphi)^{-\gamma}}{\mu_{\alpha,\beta}^{+}(B(\theta,|\theta-\varphi|))}.
\end{align*}
\item[(b)]
If $\ep_1+\ep_2+K+2\tau-\gamma\le M+N+1$, then
\begin{align*}
\Big\|\sup_{\widetilde{\theta}\in I(\theta,\theta')} \big|\partial_{\varphi}^{\ep_2}\partial_{\theta}^{\ep_1}\partial_{t}^{K}H_{t}^{\alpha+\tau,\beta+\tau}(\widetilde{\theta}, \varphi)\big|\Big\|_{\mathbb{B}}\lesssim \frac{(\theta+\varphi)^{-\gamma}(\pi-\theta+\pi-\varphi)^{-\gamma}}{|\theta-\varphi|\,\mu_{\alpha,\beta}^{+}(B(\theta,|\theta-\varphi|))}.
\end{align*}
\end{itemize}

After interchanging the roles of $\theta$ and $\varphi$ on the left-hand side of the
bound in (b), the estimate holds uniformly in $\theta,\varphi,\varphi'\in(0,\pi)$ satisfying $2|\varphi-\varphi'|<|\theta-\varphi|.$ 
\end{lem}
\begin{proof}
We first show (b) with $\tau=0$. Assume that
 $\ep_1+\ep_2+K\le M+N+1$. By \eqref{PJker} and Minkowski's integral inequality
\begin{align*}
&\Big\|\sup_{\widetilde{\theta}\in I(\theta,\theta')}\big|\partial_{\varphi}^{\ep_2}\partial_{\theta}^{\ep_1}\partial_{t}^{K}H_{t}^{\alpha,\beta}(\widetilde{\theta}, \varphi)\big|\Big\|_{\mathbb{B}}\\
&\lesssim \iint d\Pi_{\alpha}(u)d\Pi_{\beta}(v)\bigg( \int_0^{\infty}\Big(\sup_{\widetilde{\theta}\in I(\theta,\theta')}\big|\partial_{\varphi}^{\ep_2}\partial_{\theta}^{\ep_1}\partial_{t}^{K}\Psi^{\ab}(t,q(u,v,\widetilde{\theta},\varphi))\big|\Big)^2 t^{2M+2N+1}dt\bigg)^{1\slash 2}.
\end{align*}
Next we split the inner integral in $t$ according to the intervals $(0,1)$ and $(1,\infty)$ and denote
the resulting integrals by $I_0$ and $I_{\infty}$, respectively. Then, using Lemma \ref{4.8} and the boundedness of $\q$, Lemma \ref{lem:comp},
and then changing the variable $t=\sqrt{\q}s$, we get
\begin{align*}
I_0 &\lesssim \int_0^1\sup_{\widetilde{\theta}\in I(\theta,\theta')}\frac{t^{2M+2N-1}dt}{(t^2+q(u,v,\widetilde{\theta},\varphi))^{2\alpha+2\beta+3+M+N+1}}\simeq \int_0^1\frac{t^{2M+2N-1}dt}{(t^2+\q)^{2\alpha+2\beta+4+M+N}}\\
&\le\frac{1}{\q^{2\alpha+2\beta+4}}\int_{0}^{\infty}\frac{s^{2M+2N-1}\,ds}{(1+s^{2})^{2\alpha+2\beta+4+M+N}}
\lesssim \frac{1}{\q^{2\alpha+2\beta+4}},
\end{align*}
provided that $2|\theta-\theta'|<|\theta-\varphi|.$
Moreover, by Lemma \ref{4.8} and the boundedness of $\q$, we have
$$
I_{\infty} \lesssim \int_1^{\infty} \frac{t^{2M+2N-1}dt}{e^{\zeta t}}\lesssim 1\lesssim 
	\frac{1}{\q^{2\alpha+2\beta+4}}
$$
for some constant $\zeta=\zeta(\alpha,\beta)>0$. Therefore, for $2|\theta-\theta'|<|\theta-\varphi|$
$$
\Big\|\sup_{\widetilde{\theta}\in I(\theta,\theta')}\big|\partial_{\varphi}^{\ep_2}\partial_{\theta}^{\ep_1}\partial_{t}^{K}H_{t}^{\alpha,\beta}(\widetilde{\theta}, \varphi)\big|\Big\|_{\mathbb{B}} \lesssim
	\iint \frac{d\Pi_{\alpha}(u)d\Pi_{\beta}(v)}{\q^{\alpha+\beta+2}}
$$
and (b) with $\tau=0$ follows from Lemma \ref{bridge}.

Proceeding essentially in the same way as above, with the aid of Lemma \ref{4.8}, we obtain 
\begin{align*}
\big\|\partial_{\theta}^{\ep_1}\partial_{t}^{K}H_{t}^{\alpha,\beta}(\theta, \varphi)\big\|_{\mathbb{B}}\lesssim	\iint \frac{d\Pi_{\alpha}(u)d\Pi_{\beta}(v)}{\q^{\alpha+\beta+3/2}}
\end{align*}
for  $\ep_1+K\le M+N$,
and thus can use Lemma \ref{bridge} to get the estimate (a) with $\tau=0$.

The case of $\tau=1$ can be treated in a similar manner, employing in addition Lemma \ref{L4.3*} and Lemma \ref{estimates}. The final assertion is justified by the symmetry of $H_{t}^{\alpha,\beta}({\theta}, \varphi)$
\end{proof}

\begin{proof}[Proof of Theorem \ref{thm:stand}; the case of $\{\partial_t^M\delta_N^{\textrm{\emph{even}}} H_t^{\ab}(\theta,\varphi)\}_{t>0}$.]

First assume that $N$ is even and write it as $N=2k_{0},$ with $k_0\ge0$. Then, by the relation \eqref{pochodne}, we see that it is enough to show that, for each $0\le k\le k_{0}$, the vector-valued kernel 
$$T_{k}^{\alpha,\beta}(\theta,\varphi):=\big\{\partial_{t}^{M+2k}H_{t}^{\alpha,\beta}(\theta,\varphi)\big\}_{t>0}$$
satisfies the standard estimates.

The growth condition \eqref{gr} for $T_{k}^{\alpha,\beta}(\theta,\varphi)$ follows directly from Lemma \ref{4.8*} (a) (taken with $\ep_1=\tau=\gamma=0,K=M+2k$).

To prove the smoothness condition \eqref{sm1} we use the Mean Value Theorem and Lemma \ref{4.8*} (b) (with $\ep_1=1,\ep_2=\tau=\gamma=0,K=M+2k$), getting for ${\theta}$ and $\theta'$ satisfying $2|\theta-\theta'|<|\theta-\varphi|$
\begin{align*}
\big\|T_{k}^{\alpha,\beta}(\theta,\varphi)-
	T_{k}^{\alpha,\beta}(\theta',\varphi)\big\|_{\mathbb{B}}
& \le |\theta-\theta'|\Big\|\sup_{\widetilde{\theta}\in I(\theta,\theta')}\big|\partial_{\theta}\partial_{t}^{M+2k}H_{t}^{\alpha,\beta}(\widetilde{\theta},\varphi)\big|\Big\|_{\mathbb{B}}\\
&\lesssim \frac{|\theta-\theta'|}{|\theta-\varphi|}\;\frac{1}{\m_{\ab}^{+}(B(\theta,|\theta-\varphi|))}.
\end{align*}
Analogous arguments lead to \eqref{sm2}, the other smoothness condition.

Considering $N$ odd, $N=2k_{0}+1$, we take into account \eqref{pochodne} and observe that it suffices to estimate for each $0\le k\le k_{0}$ the vector-valued kernel
$$\mathcal{T}_{k}^{\alpha,\beta}(\theta,\varphi):=\big\{\partial_{\theta}\partial_{t}^{M+2k}H_{t}^{\alpha,\beta}(\theta,\varphi)\big\}_{t>0}.$$
This, however, can be done essentially by repeating the arguments used for the case of $N$ even. We leave the details to interested readers.
\end{proof}

Next we treat the kernel $\{\partial_t^M\ko\widetilde{H}_t^{\ab}(\theta,\varphi)\}_{t>0}$.

\begin{proof}[Proof of Theorem \ref{thm:stand}; the case of $\{\partial_t^M\delta_N^{\textrm{\emph{odd}}}\widetilde{H}_t^{\ab}(\theta,\varphi)\}_{t>0}$.]

We begin with considering the case of $N$ even and write it as $N=2k_{0}$ with $k_0\ge0$. By iterating \eqref{delta2} we see that it is enough to show that for each $0\le k\le k_{0}$ the vector-valued kernel
$$\widetilde{T}_{k}^{\alpha,\beta}(\theta,\varphi)=\big\{\partial_{t}^{M+2k}\widetilde{H}_{t}^{\alpha,\beta}(\theta,\varphi)\big\}_{t>0} $$
satisfies the standard estimates.

The growth estimate is straightforward. Applying Lemma \ref{estimates} (b) (with $\widetilde{\theta}=\theta$) and Lemma \ref{4.8*} (a) (with $\ep_1=0,\tau=1,\gamma=2,K=M+2k$) we get
\begin{align*}
\big\|\widetilde{T}_{k}^{\alpha,\beta}(\theta,\varphi)\big\|_{\mathbb{B}}&= 
\frac{\sin\theta\,\sin\varphi}{4}\,\big\|H_{t}^{\alpha+1,\beta+1}(\theta,\varphi)\big\|_{\mathbb{B}}\lesssim\frac{1}{\mu_{\alpha,\beta}^{+}(B(\theta,|\theta-\varphi|))}.
\end{align*}  

We pass to the smoothness conditions. For symmetry reasons, it is enough to show \eqref{sm1}. Using the Mean Value Theorem  together with Lemma \ref{estimates} (b) we obtain for $2|\theta-\theta'|<|\theta-\varphi|$
\begin{align*}
&\big\|\widetilde{T}_{k}^{\alpha,\beta}(\theta,\varphi)-\widetilde{T}_{k}^{\alpha,\beta}(\theta',\varphi)\big\|_{\mathbb{B}}\\
&\le|\theta-\theta'|\bigg[(\theta+\varphi)^{2}(\pi-\theta+\pi-\varphi)^{2}\,\Big\|\sup_{\widetilde{\theta}\in I(\theta,\theta')}\big|\partial_{\theta}\,\partial_{t}^{M+2k}H_{t}^{\alpha+1,\beta+1}(\widetilde{\theta},\varphi)\big|\Big\|_{\mathbb{B}}\\
&\quad+(\theta+\varphi)(\pi-\theta+\pi-\varphi)\,\Big\|\sup_{\widetilde{\theta}\in I(\theta,\theta')}\big|\partial_{t}^{M+2k}H_{t}^{\alpha+1,\beta+1}(\widetilde{\theta},\varphi)\big|\Big\|_{\mathbb{B}}\bigg].
\end{align*}
Now, we apply Lemma \ref{4.8*} (b) (with  $\ep_1=1,\ep_2=0,\tau=1,\gamma=2,K=M+2k$ in case of the first of the above components and $\ep_1=\ep_2=0,\tau=\gamma=1,K=M+2k$ in case of the second one), obtaining
\begin{align*}
&\big\|\widetilde{T}_{k}^{\alpha,\beta}(\theta,\varphi)-\widetilde{T}_{k}^{\alpha,\beta}(\theta',\varphi)\big\|_{\mathbb{B}}\lesssim\frac{|\theta-\theta'|}{|\theta-\varphi|}\,\frac{1}{\mu_{\alpha,\beta}^{+}(B(\theta,|\theta-\varphi|))}
\end{align*}
for $2|\theta-\theta'|<|\theta-\varphi|$.

Consider now $N$ odd, say $N=2k_{0}+1$ with $k_0\ge0$. The kernel we need to estimate is, see \eqref{delta2},
\begin{equation*}
\big\{\partial_{t}^{M}\ko\widetilde{H}_{t}^{\alpha,\beta}(\theta,\varphi)\big\}_{t>0}=\bigg\{\sum_{j=0}^{k_{0}}c_{j}\,\delta^{*}\partial_{t}^{M}\partial_{t}^{2(k_{0}-j)}\widetilde{H}_{t}^{\alpha,\beta}(\theta,\varphi)\bigg\}_{t>0}.
\end{equation*}
Thus it is enough to show that for each $0\le k\le k_{0}$ the kernel
$$\widetilde{\mathcal{T}}_{k}^{\alpha,\beta}(\theta,\varphi)=\big\{\delta^{*}\partial_{t}^{M+2k}\widetilde{H}_{t}^{\alpha,\beta}(\theta,\varphi)\big\}_{t>0}
$$ satisfies the standard estimates.
A direct computation reveals that
\begin{align*}
\|\widetilde{\mathcal{T}}_{k}^{\alpha,\beta}(\theta,\varphi)\|_{\mathbb{B}}&\lesssim\sin\theta\,\sin\varphi\,\|\partial_{\theta}\partial_{t}^{M+2k}H_{t}^{\alpha+1,\beta+1}(\theta,\varphi)\|_{\mathbb{B}}\\ 
&\quad+\sin\varphi\,\|\partial_{t}^{M+2k}H_{t}^{\alpha+1,\beta+1}(\theta,\varphi)\|_{\mathbb{B}}.
\end{align*}
An application of Lemma \ref{4.8*} (a) to both of the above terms (with $\ep_1=\tau=1,\gamma=2,K=M+2k$ in case of the first one and with $\ep_1=0,\tau=\gamma=1,K=M+2k$ to the second one) gives the growth condition.

To show \eqref{sm1}, notice that the Mean Value Theorem, \eqref{delta2} and Lemma \ref{estimates} (b) lead to the bound
\begin{align*}
&\big\|\widetilde{\mathcal{T}}_{k}^{\alpha,\beta}(\theta,\varphi)-\widetilde{\mathcal{T}}_{k}^{\alpha,\beta}(\theta',\varphi)\big\|_{\mathbb{B}}\\
&\lesssim |\theta-\theta'|(\theta+\varphi)^{2}(\pi-\theta+\pi-\varphi)^{2}\bigg[\Big\|\sup_{\widetilde{\theta}\in I(\theta,\theta')}\big|\partial_{t}^{M+2k+2}H_{t}^{\alpha+1,\beta+1}(\widetilde{\theta},\varphi)\big|\Big\|_{\mathbb{B}}\\
&\quad+\Big\|\sup_{\widetilde{\theta}\in I(\theta,\theta')}\big|\partial_{t}^{M+2k}H_{t}^{\alpha+1,\beta+1}(\widetilde{\theta},\varphi)\big|\Big\|_{\mathbb{B}}\bigg].
\end{align*}
Applying Lemma \ref{4.8*} (b) to each component in the above sum (with $\ep_1=\ep_2=0,\tau=1,\gamma=2,K=M+2k+2$ and with $\ep_1=\ep_2=0,\tau=1,\gamma=2,K=M+2k$, respectively) gives the desired conclusion.

It remains to prove \eqref{sm2}.
By the Mean Value Theorem and Lemma \ref{estimates} (b)
\begin{align*}
&\big\|\widetilde{\mathcal{T}}_{k}^{\alpha,\beta}(\theta,\varphi)-\widetilde{\mathcal{T}}_{k}^{\alpha,\beta}(\theta,\varphi')\big\|_{\mathbb{B}}\\
&\lesssim |\varphi-\varphi'|\big\|\partial_{\varphi}\delta^{*}\partial_{t}^{M+2k}\widetilde{H}_{t}^{\alpha,\beta}(\theta,\widetilde{\varphi})\big\|_{\mathbb{B}}\\
&\lesssim|\varphi-\varphi'|\bigg[\Big\|\sup_{\widetilde{\varphi}\in I(\varphi,\varphi')}\big|\partial_{t}^{M+2k}H_{t}^{\alpha+1,\beta+1}(\theta,\widetilde{\varphi})\big|\Big\|_{\mathbb{B}}\\
&\quad+(\theta+\varphi)(\pi-\theta+\pi-\varphi)\Big\|\sup_{\widetilde{\varphi}\in I(\varphi,\varphi')}\big|\partial_{\theta}\partial_{t}^{M+2k}H_{t}^{\alpha+1,\beta+1}(\theta,\widetilde{\varphi})\big|\Big\|_{\mathbb{B}}\\ &\quad+(\theta+\varphi)(\pi-\theta+\pi-\varphi)\Big\|\sup_{\widetilde{\varphi}\in I(\varphi,\varphi')}\big|\partial_{\varphi}\partial_{t}^{M+2k}H_{t}^{\alpha+1,\beta+1}(\theta,\widetilde{\varphi})\big|\Big\|_{\mathbb{B}}\\
&\quad+(\theta+\varphi)^{2}(\pi-\theta+\pi-\varphi)^{2}\Big\|\sup_{\widetilde{\varphi}\in I(\varphi,\varphi')}\big|\partial_{\theta}\partial_{\varphi}\partial_{t}^{M+2k}H_{t}^{\alpha+1,\beta+1}(\theta,\widetilde{\varphi})\big|\Big\|_{\mathbb{B}}\bigg].
\end{align*}
Now observe that each of the above terms can be treated by means of Lemma \ref{4.8*} with suitably chosen parameters. Consequently, we conclude that for $2|\varphi-\varphi'|\le|\theta-\varphi|$
\begin{align*}
\big\|\widetilde{\mathcal{T}}_{k}^{\alpha,\beta}(\theta,\varphi)-\widetilde{\mathcal{T}}_{k}^{\alpha,\beta}(\theta,\varphi')\big\|_{\mathbb{B}}\lesssim \frac{|\varphi-\varphi'|}{|\theta-\varphi|}\,
	\frac{1}{\m_{\ab}^{+}(B(\theta,|\theta-\varphi|))}.
\end{align*}

The proof of the case of $\{\partial_t^M\ko\widetilde{H}_t^{\ab}(\theta,\varphi)\}_{t>0}$ in Theorem \ref{thm:stand} is finished.
\end{proof}

Finally, we estimate the kernels related to Laplace and Laplace-Stieltjes transform type multipliers.
\begin{proof}[Proof of Theorem \ref{thm:stand}; the cases of $M_{\phi}^{\alpha,\beta}(\theta,\varphi)$ and $\widetilde{M}_{\phi}^{\alpha,\beta}(\theta,\varphi)$] Since $\phi$ is bounded, we have
\begin{align*}
|M_{\phi}^{\alpha,\beta}(\theta,\varphi)|\lesssim\int_{0}^{\infty}|\partial_{t}H_{t}^{\alpha,\beta}(\theta,\varphi)|dt\lesssim\iint\,d\Pi_{\alpha}(u)d\Pi_{\beta}(v)\int_{0}^{\infty}|\partial_{t}\Psi^{\ab}(t,\q)|\,dt.
\end{align*}
Now we split the inner integral in $t$ into two parts according to small and large $t$ and suitably bound them with the aid of Lemma \ref{4.8} obtaining
$$|M_{\phi}^{\alpha,\beta}(\theta,\varphi)|\lesssim\iint\frac{d\Pi_{\alpha}(u)d\Pi_{\beta}(v)}{\q^{\alpha+\beta+3/2}}.$$
From here the growth estimate follows by Lemma \ref{bridge}.

Next, we show the gradient condition for $M_{\phi}^{\alpha,\beta}(\theta,\varphi)$. For symmetry reasons it is enough to estimate only the derivative in $\theta$.
We have
\begin{align*}
|\partial_{\theta}M_{\phi}^{\alpha,\beta}(\theta,\varphi)|\lesssim\iint d\Pi_{\alpha}(u)d\Pi_{\beta}(v)\int_{0}^{\infty}|\partial_{\theta}\partial_{t}\Psi^{\ab}(t,\q)|dt.
\end{align*}
Proceeding as in the proof of the growth condition we get
$$|\partial_{\theta}M_{\phi}^{\alpha,\beta}(\theta,\varphi)|\lesssim\iint\frac{d\Pi_{\alpha}(u)d\Pi_{\beta}(v)}{\q^{\alpha+\beta+2}}$$
and using Lemma \ref{bridge} we arrive at the desired bound.

The estimates for $\widetilde{M}_{\phi}^{\alpha,\beta}(\theta,\varphi)$ are also straightforward. The growth bound follows from the growth condition for $M_{\phi}^{\alpha+1,\beta+1}(\theta,\varphi)$ and Corollary \ref{cor:compare}.
To show \eqref{sm} we write
\begin{align*}
|\partial_{\theta}\widetilde{M}_{\phi}^{\alpha,\beta}(\theta,\varphi)|&\lesssim\sin\theta\sin\varphi\,\bigg|\int_{0}^{\infty}\partial_{\theta}\partial_{t}H_{t}^{\alpha+1,\beta+1}(\theta,\varphi)\phi(t)\,dt\bigg|\\
&\quad+\sin\varphi\,\bigg|\int_{0}^{\infty}\partial_{t}H_{t}^{\alpha+1,\beta+1}(\theta,\varphi)\phi(t)\,dt\bigg|.
\end{align*}
Using \eqref{sm} for $M_{\phi}^{\alpha+1,\beta+1}(\theta,\varphi)$ and Corollary \ref{cor:compare} we see that the first term here satisfies the gradient estimate. To bound the remaining term we use \eqref{gr} for $M_{\phi}^{\alpha+1,\beta+1}(\theta,\varphi)$ and then apply Corollary \ref{cor:compare} and Lemma \ref{estimates} (a). To conclude, observe that the analogous bound for the $\partial_{\varphi}$ component in \eqref{sm} follows by symmetry.
\end{proof}

\begin{proof}[Proof of Theorem \ref{thm:stand}; the cases of $M_{\nu}^{\alpha,\beta}(\theta,\varphi)$ and $\widetilde{M}_{\nu}^{\alpha,\beta}(\theta,\varphi)$]
We first deal with ${M}_{\nu}^{\alpha,\beta}(\theta,\varphi).$ Recall that $\ab\ge-1/2$ and so $\alpha+\beta+1\ge0$.
Taking into account that
$$\int_{(0,\infty)}e^{-t(\alpha+\beta+1)/2}\,d|\nu|(t)<\infty,$$
it is clear that the task of showing the growth condition will be accomplished once we prove the estimate
$$e^{t(\alpha+\beta+1)/2}{H}_{t}^{\alpha,\beta}(\theta,\varphi)\lesssim\frac{1}{\m_{\ab}^{+}(B(\theta,|\theta-\varphi|))},\qquad t>0,\quad \theta,\varphi\in (0,\pi).$$
Observe that (see the proof of \cite[Theorem 2.4]{NoSjogren}, the case of $\mathcal{H}_*^{\ab}$) we have
$$\Psi_{t}^{\alpha,\beta}(\theta, \varphi)=\frac{\sinh\frac{t}{2}}{(\cosh\frac{t}{2}-1+\q)^{\alpha+\beta+2}}\lesssim\frac{1}{\q^{\alpha+\beta+3/2}}.$$
Since $e^{t(\alpha+\beta+1)/2}\Psi_{t}^{\alpha,\beta}(\theta,\varphi)\simeq 1\lesssim \q^{-(\alpha+\beta+3/2)}$ for $t$ large, it follows that 
$$e^{t(\alpha+\beta+1)/2}H_t^{\ab}(\theta,\varphi)\lesssim\iint\frac{d\Pi_{\alpha}(u)d\Pi_{\beta}(v)}{\q^{\alpha+\beta+3/2}},\qquad t>0,\quad \theta,\varphi\in(0,\pi).$$
The growth condition follows now from Lemma \ref{bridge}.

To prove the gradient condition we use the bound
$$|\partial_{\theta}H_{t}^{\alpha,\beta}(\theta,\varphi)|\lesssim\sinh\frac{t}{2}\,\iint\frac{d\Pi_{\alpha}(u)d\Pi_{\beta}(v)}{(\cosh\frac{t}{2}-1+\q)^{\alpha+\beta+5/2}}, $$ which follows from Lemma \ref{trig}. Since
$$\frac{e^{t(\alpha+\beta+1)/2}\sinh\frac{t}{2}}{(\cosh\frac{t}{2}-1+\q)^{\alpha+\beta+5/2}}\lesssim\frac{1}{\q^{\alpha+\beta+2}},\qquad t>0,$$
with the aid of Lemma \ref{bridge} we arrive at the estimate
$$\int_{(0,\infty)}e^{t(\alpha+\beta+1)/2}\,|\partial_{\theta}H_{t}^{\alpha,\beta}(\theta,\varphi)|e^{-t(\alpha+\beta+1)/2}\,d|\nu|(t)\lesssim\frac{1}{|\theta-\varphi| \m_{\ab}^{+}(B(\theta,|\theta-\varphi|))}.$$
This together with the symmetry of the kernel implies \eqref{sm} for ${M}_{\nu}^{\alpha,\beta}(\theta,\varphi).$ 
This finishes the proof in the case of ${M}_{\nu}^{\alpha,\beta}(\theta,\varphi).$

The standard estimates for $\widetilde{M}_{\nu}^{\alpha,\beta}(\theta,\varphi)$ essentially follow from those for ${M}_{\nu}^{\alpha,\beta}(\theta,\varphi)$, in a similar manner as for the Laplace transform type multipliers. We leave details to the reader.

\end{proof}

The proof of Theorem \ref{thm:stand} is complete.

\end{document}